\documentclass{amsart}
\usepackage{amsmath,amssymb,amsthm,amscd}
\usepackage[a4paper,text={140mm,210mm},centering,headsep=5mm,footskip=10mm]{geometry} 
\usepackage[matrix,arrow, curve]{xy}
\numberwithin{equation}{section}
\usepackage{color}
\usepackage{graphicx}
\usepackage[hypertexnames=false,linkcolor=black, colorlinks=true, citecolor=black, filecolor=black,urlcolor=black]{hyperref}

\theoremstyle{plain}
\newtheorem{theorem}{Theorem}
\newtheorem{conj}[theorem]{Conjecture}

\newtheorem{Thm}{Theorem}[section]
\newtheorem{Lem}[Thm]{Lemma}
\newtheorem{Prop}[Thm]{Proposition}

\newtheorem{Facts}[Thm]{Facts}

\theoremstyle{definition}
\newtheorem{Def}[Thm]{Definition}
\newtheorem{Rk}[Thm]{Remark}
\newtheorem{Ex}[Thm]{Example}

\newtheorem{Obs}[Thm]{Observation}
\newtheorem{Constr}[Thm]{Construction}
\newtheorem{Proc}[Thm]{Procedure}

\newcommand{\RSP}{regular system of parameters }
\newcommand{\wrt}{with respect to }

\newcommand{\half}{{\frac{1}{2}}}
\newcommand{\gqz}{{\geq 0}}
		
\newcommand{\sg}{\sigma}		
\newcommand{\lm}{\lambda}
\newcommand{\car}[1]{\mathrm{char}(#1)}		
\newcommand{\dell}[1]{ \frac{\partial }{ \partial #1 }}		
\newcommand{\cB}{\mathcal{B}}

\newcommand{\ID}{\mathbb{ D }}
\newcommand{\IE}{\mathbb{ E }}
\newcommand{\IF}{\mathbb{ F }}
\newcommand{\IG}{\mathbb{ G }}

\newcommand{\IQ}{\mathbb{ Q }}
\newcommand{\IR}{\mathbb{ R }}
\newcommand{\IT}{\mathbb{ T }}
\newcommand{\IZ}{\mathbb{ Z }}

\newcommand{\IRid}{{\IR}\mathrm{id}}
\newcommand{\IDir}{{\ID}\mathrm{ir}}

\newcommand{\cD}{\mathcal{ D }}
\newcommand{\cO}{\mathcal{ O }}

\newcommand{\ord}{{\mathrm{ord}}}
\newcommand{\Rid}{{\mathrm{Rid}}}
\newcommand{\Dir}{{\mathrm{Dir}}}
\newcommand{\Diff}{{\mathrm{Diff}}}
\newcommand{\Sing}{{\mathrm{Sing}}}
\newcommand{\Spec}{{\mathrm{Spec}}}

\newcommand{\TC}{T}
\newcommand{\ITC}{\mathbb{\TC}}

\newcommand{\leftmapsto}{\mathrel{\reflectbox{\ensuremath{\longmapsto}}}}

\begin{document}

\title{
Partial local resolution by characteristic zero methods}

\author{Bernd Schober}
\thanks{Supported by Research Fellowships of the Deutsche Forschungsgemeinschaft (SCHO 1595/1-1 and SCHO 1595/2-1).}
\address{Bernd Schober\\
	The Fields Institute\\ 
	222 College Street\\ 
	Toronto ON M5T 3J1\\ 
	Canada\\ 
	and 
	University of Toronto\\
	Department of Mathematics\\
	40 St. George Street\\
	Toronto, Ontario, M5S 2E4
	\\
	Canada}
\curraddr{
Institut f\"ur Algebraische Geometrie\\
Leibniz Universit\"at Hannover\\
Welfengarten 1\\
30167 Hannover\\
Germany}
\email{schober@math.uni-hannover.de}
%\email{schober.math@gmail.com} 


\begin{abstract}
	We discuss to what extent the local techniques of resolution of singularities over fields of characteristic zero can be applied to improve singularities in general.
	For certain interesting classes of singularities, this leads to an embedded resolution via blowing ups in regular centers.
	We illustrate this for generic determinantal varieties.
	The article is partially expository and is addressed to non-experts who aim to construct resolutions for other special classes of singularities in positive or mixed characteristic. 
\end{abstract}

\maketitle

%
%
%
%
%
%
%
%		Intro
%
%
%
%
%
%
%

\section*{Introduction}
\label{Intro}
The study of the geometry of a scheme may be quite hard if singularities appear.
The goal of resolution of singularities is to reveal the information hidden in the singularities.

\begin{conj}[Embedded Resolution of Singularities]
	\label{Conj:ERS}
	Let $ X \subset Z $ be a reduced closed subscheme of a quasi-excellent regular Noetherian scheme $ Z $.
	There exist a regular scheme $ Z' $ and a proper, birational morphism $ \pi : Z' \to Z $ such that
	\begin{enumerate}
		\item 
		the strict transform $ X'  \subset Z' $ of $ X $ in $ Z' $
		is regular,
		
		\item 
		$ \pi $ is an isomorphism outside of the singular locus,
		$ \pi^{-1}(Z \setminus \Sing(X)) \cong Z \setminus \Sing(X) $,
		
		\item 
		$ \pi^{-1} (\Sing(X))_{red} $ is a simple normal crossing divisor 
		(i.e., each irreducible component is regular and they intersect transversally)
		which intersects $ X' $ transversally.
	\end{enumerate}
\end{conj}

Over fields of characteristic zero the conjecture is true due to Hironaka's celebrated theorem \cite{Hiro64}.
Several people have studied and improved his result so that nowadays there exists a quite accessible proof for canonical constructive resolution of singularities in characteristic zero, 
\cite{BMheavy, BMlight,  AnaSantiOrlando05, Cu, SantiHerwig,EleonoreHerwig,HerwigChar0,  Kollar, BerndBM, TemkinQEchar0, OrlandoConstr, OrlandoPatch, Jarek}.
The latter means that $ Z' \to Z $ is obtained by a finite sequence of blowing ups in canonical regular centers.
Let us mention that there also exist implementations of the resolution algorithm in characteristic zero \cite{BodnarSchicho, AnnePfister}.

In contrast to this, 
there are only results in small dimensions,
\cite{AbhSurfaces, Abh5, Abh6, Abh7, Abh8, AngOrlDim2, CosToho, CGO, CJS, CS_Dim2,Cut3fold,CutSurf, HerwigWagner, HiroBowdoin, IFPdim2new, LipDim2},
or for special cases,
\cite{Complexity, BMbinomial, Blanco1, Blanco2, VincentPetite, CPsmallmulti,KollarToroidal,OrlandoNew},
in positive or mixed characteristic. 
For threefolds, Cossart and Piltant proved the existence of a global birational model $ X' \to X $ without an embedding  (first over differentially finite fields \cite{CP1, CP2} and then in the arithmetic case \cite{CPsmallmulti, CPmixed, CPmixed2}).
The general case of resolution in dimension four and higher, and embedded resolution in dimension three remain very difficult problems.
For more details on difficulties that need to be overcome, we refer to \cite{HerwigCharPos, HerwigStefan1, HerwigStefan2, Moh1}. 

A local variant of resolution of singularities, so called local uniformization, is studied in \cite{KnafKuhlmann, Moh2, NovaMarkReduction1,NovaMarkReduction2,Jean-Christophe,BernardAbhy,TemkinInsepLU}, for example.
Further, if one requires $ X' \to X $ only to be generically finite and not birational, de Jong \cite{deJong} showed the existence of such a model in any dimension.
Later, this has been refined by Gabber \cite{IllusieTemkin} and by Temkin \cite{TemkinAnnals}.

The purpose of the present article is to show how central techniques of characteristic zero are interesting in positive or mixed characteristic and can be developed so that they can be applied to certain classes of singularities.
Non-trivial examples are the resolution of varieties defined by binomial ideals (\cite{BMbinomial}, \cite{Blanco1}, \cite{Blanco2}), or of generic determinantal varieties (section 5).
\\[3pt]
\indent 
The local resolution data of $ X \subset Z $ at a point $ x \in X $ can be encoded in a {\em pair} $ \IE = (J,b) $ on the regular local ring $ R = {\cO}_{Z,x} $, for an ideal $ J \subset R $ (connected to the ideal locally defining $ X $ at $ x $)
and a positive rational number $ b \in \IQ_+ $.
In particular, we are interested in its singular locus $ \Sing(\IE) $ which is the set of those points in $ X $, where $ J $ has at least order $ b $ (Definition \ref{Def:idexp}).
The goal is to find a sequence of blowing ups such that the transform of $ \IE $ has no points of order $ b $.
(In section 1, we explain more details on this).

It is natural to identify pairs which undergo the same resolution process.
Thus we obtain an equivalence relation on the set of pairs on $ R $ for which an equivalence class is called {\em idealistic exponent}. 
In Theorem \ref{Thm:ReductionIntro} we address the task of finding an appropriate representative of an idealistic exponent which reveals the nature of the singularities of $ \IE $.

Before we can state the results, we need to introduce some technical objects that play a central role in the article.
Suppose $ \Sing(\IE) \neq \varnothing $.
Let $ M $ be the maximal ideal of $ R $.
A first approximation for
the singularity of 
$ \IE= (J,b) $ at $ M $ is given by the \textit{tangent cone pair} $ \IT_M( \IE) = (In_M(J,b),b) $ (Definition \ref{Def:idealistic_tangent_cone_etc}).
It is a pair on the  associated graded ring of $ R $, $ gr_M(R) := \bigoplus_{i \geq 0} M^i/M^{i+1} $
(which is isomorphic to a polynomial ring over the residue field),
and $ In_M(J,b) $ is a homogeneous ideal generated by the $ b $-initial forms of elements of $ J $ (Definition \ref{Def:Tangent}).
From this, we obtain a cone $ C := \Spec (gr_M(R)/ In_M(J,b)) $.

Associated to $ C $ there is the \textit{directrix} $ \Dir_M(\IE) $ of $ \IE $ (at $ M $).
This is the largest sub-vector space $ V $ of $ \Spec(gr_M(R)) $ leaving the cone $ C $ stable under translation, $ C + V  = C $
(Defintion \ref{Def:Dir_Rid}(1)).
The significance of the directrix appears if the characteristic of the residue field is either zero or larger than $ \dim(X)/2+1 $.
In this case, the points, where the singularity does not improve after blowing up a permissible center, have to lie in a projective space associated to the directrix
(see \cite{CJS} Theorem 2.14 which is based on work by Hironaka \cite{HiroCertain} Theorem IV, \cite{HiroAdditive} Theorem 2, and Mizutani \cite{Miz}). 
Moreover, over fields of characteristic zero, the directrix has a strong connection to subvarieties of maximal contact (Definition \ref{Def:max_contact}).
The latter is a regular subvariety $ H \subset \Spec(R) $ that contains $ \Sing(\IE) $ and the inclusion remains true for the respective transforms under permissible blowing ups, i.e., $ \Sing(\IE') \subset H' $.
This allows to reduce the resolution problem to one in a smaller dimensional ambient space, namely in $ H $.
This is a crucial step in proving resolution of singularities in characteristic zero.

While maximal contact locally always exists in characteristic zero, 
this is not true in general.
This leads to the notion of the \textit{ridge} (or {\em fa\^ite} in French literature) 
$ \Rid_M(\IE)$ of $ \IE$ (at $ M $), a generalization of the directrix
(see Giraud's work \cite{GiraudEtude, GiraudMaxPos}).
The ridge associated to the cone $ C $ above is the largest additive subgroup of $ \Spec(gr_M(R)) $ leaving $ C $ stable under translation
(Definition \ref{Def:Dir_Rid}(2)).
Without any restriction on the characteristic, the ridge provides information on the locus of points where the singularity does not improve after blowing up a permissible center.
Unfortunately, one only achieves ``singular subvarieties of maximal contact".
We explain this in more detail below since one of the consequences of Theorem \ref{Thm:ReductionIntro} is a variant of this in the idealistic setting.

	The ridge $ \Rid_M (\IE) $ is defined by certain additive polynomials $ \sigma_1, \ldots, \sigma_s $,
	where $ \sigma_i $ is homogeneous of degree $ q_i := p^{d_i} $ and $ q_1 \leq \ldots \leq q_s $,
	for $ p = \car{k} \geq 0 $.
	(In characteristic zero, the only additive polynomials are those homogeneous of degree one and hence $ q_i = 1 $ for all $ i $).
	This provides the pair $ \IRid_M (\IE) := (\sigma_1, q_1) \cap \ldots \cap (\sigma_s ,q_s) $ living in $ gr_M(R) $.
	(For the definition of an intersection of pairs see \eqref{eq:intersect_pairs}).
	For a suitable choice of $ ( \sigma ) $, we have that for every $ \sigma_i $ there exists a differential operator $ \cD_i $ on $ gr_M(R) $ of order $ b - q_i $ and some $ F_i \in In_M(J,b) $ such that
\begin{equation}
\label{eq:diff_sigma}
		\cD_i F_i = \sigma_i.
\end{equation}
	The $ \cD_i $ are given by Hasse-Schmidt derivatives $ \dell{Y^B} $ (see \eqref{eq:HasseSchmdit}), where $ ( Y ) = (Y_1, \ldots, Y_r ) $ is a subset of variables in $ gr_M(R) $ defining $ \Dir_M(\IE) $.
	(For the meaning of the latter condition, see Definition \ref{Def:(3)_defines directrix_ridge}(3)).

	\begin{theorem}[Theorem \ref{Thm:Equiv}]
		\label{Thm:EquivIntro} 
		Let $ R $ be a regular local ring with maximal ideal $ M $ and $ \IE $ be a pair on $ R $.
		Suppose that $ \Sing( \IE) \neq \varnothing $.
		Then the pairs $  \IT_M (\IE) $  and $ \IRid_M(\IE) $ are equivalent.		
	\end{theorem}
	
	In other words, the theorem states that all information on the singularity of $ \IE $ at $ M $ that we can obtain from the tangent pair $ \IT_M (\IE) $ is concentrated in $ \IRid_M(\IE) $.

	One of the key ingredients for the proof is \eqref{eq:diff_sigma}.
	In order to obtain the desired representative for the idealistic exponent corresponding to $ \IE $, 
	we want to translate the arguments for the proof of the previous theorem from the graded ring $ gr_M(R) $ to $ R $. 
	For this, we need to be able to lift the differential operators $ \cD_i $ (on $ gr_M(R) $) to $ R $.
	Hence, we have to impose the following assumption:
	Let $ ( u, y  ) = (u_1, \ldots, u_e; y_1, \ldots, y_r ) $ be a regular system of parameters for $ R $ such that $ ( Y ) $ defines the directrix of $ \IE $, where $ Y_j := y_j \mod M^2 \in gr_M(R) $, for $ 0 \leq j \leq r $.
	We say {\em condition $(\cD)$ holds for $ ( \IE,  R ) $} if
	\begin{equation}\tag{${\mathcal{D}}$}
		\label{condD}
		\left\{ \;\;\,
		\parbox[c]{320pt}{ 
			the Hasse-Schmidt derivatives 
			$ \dell{Y^B} : gr_M(R) \to gr_M(R) $
			lift to Hasse-Schmidt derivatives 
			$ \dell{y^B} : R \to R $ 
			on $ R $, for every $ B \in \IZ^{r}_\gqz $ with $ |B| < b $. 
		} 
		\right.
	\end{equation}
	
	%\medskip 
	%
	\noindent 
	For example, this condition is fulfilled for any $ \IE $ if $ R = S[[y]]$, or if $ R = S[y]_M $, where $ S $ is a regular local ring with maximal ideal $ N $ and $ M := N + \langle y \rangle
	\subset S[y] $.
	(Note that this includes the case $ R = k[x_1, \ldots, x_n ]_{\langle x_1, \ldots, x_n \rangle} $ for a field $ k $).
	Another large class of examples is provided by Nomura's theorem (see \cite{MatsumuraRing} Theorem 30.6 p.~237) which implies \eqref{condD} 
	if $ R $ is equi-characteristic of dimension $ n $ fulfilling that the module of $ k $-derivations $ \mbox{Der}_k (R) $ is a free $ R $-module of rank $ n $, where $ k $ is a quasi-coefficient field of $ R $
	(see also the cited theorem for further equivalent conditions).

	\begin{theorem}[Theorem \ref{Thm:Reduction}]
		\label{Thm:ReductionIntro}
			Let $ R $ be a regular local ring with maximal ideal $ M $ and residue characteristic $ p = \car{R/M} \geq 0 $.
			Let $ \IE $ be a pair on $ R $
			and, using the above notations, let $ \IRid_M(\IE) = \bigcap_{i=1}^s (\sigma_i , q_i ) $.
			Suppose that $ \Sing( \IE) \neq \varnothing $ and that condition \eqref{condD} holds for $ (\IE, R) $.
		Then $ \IE $ is equivalent to a pair of the form $ \IG \cap \ID_+ $ with 
		\begin{enumerate}
			\item 
			$ \IG = \bigcap\limits_{i=1}^s ( g_i, q_i ) $,
			for some $ g_i \in  R $ such that $ g_i \equiv  \sigma_i \mod  M^{q_i + 1} $,
			and
			
			\item 
			$ \ID_+ = (I, c) $, for some ideal $ I \subset R $ and $ c \in \IZ_+ $ with $ \ord_{M} (I) > c $.
		\end{enumerate}
	\end{theorem} 

	If $ R $ contains a perfect field $ k $ this already appears in \cite{HiroIdExp}, where it is called the Tschirnhausen decomposition of $ \IE $ 
	(loc.~cit.~Definition 3 and Theorem 4).

	Theorem \ref{Thm:ReductionIntro} is a translation of Theorem \ref{Thm:EquivIntro} from the graded ring to $ R $.
	In particular, we see how the ridge yields ``singular maximal contact" in the idealistic framework:	
	Using facts about pairs, Theorem \ref{Thm:ReductionIntro} implies 
	that $ \Sing(\IE) \subset \bigcap_i \Sing(g_i, q_i) $ and 
	this inclusion is stable under permissible blowing ups, since the equivalence is.
	Hence we get some control on the behavior of $ \Sing (\IE) $ in the sense that $ V (g_1, \ldots, g_s ) $ (which we obtain from condition (1)) is a subvariety of singular maximal contact.

	Secondly, we can deduce from the theorem (more precisely, by studying the pair in (1)) 
	whether it is possible to reduce the resolution problem to one in smaller dimension:
	If we have $ q_1 = 1 $, then $ V(g_1) = \Spec(R/g_1) $ is regular and has maximal contact
(in the characteristic zero sense).
	Therefore, using the technique of coefficient pairs (Definition \ref{Def:IdCoeffExp}), we can reduce the resolution problem for $ \IE $ on $ R $ to the one for 
	some pair $ \IE_*  = (J_*, b_* )$ on $ R/g_1 $.
	Since $ \dim (R/g_1) < \dim (R) $, it is possibly easier to construct a sequence of blowing ups in regular centers resolving $ \IE_* $.
	If we find such a sequence, the maximal contact condition implies that the same sequence resolves $ \IE $.
	In Procedure \ref{CharacterizationProcedure}, we describe a more general version of the reduction principle.
	
	As a third consequence, we see which kind of pairs need to be understood in order to make progress on resolution of singularities:
	We can read off when a reduction in the sense of characteristic zero is not possible anymore.
	For example, as it is known to experts, an important case that has to be understood is how to resolve hypersurfaces of the form
	$$ 
		V( y^{q} + h(u_1,u_2,u_3,y) ), 
		\ \ \ \mbox{ for } h \in M^{q+1} 
		\ \ \mbox{ and } \ \ q = p^d, \ d \geq 2.
	$$

	\smallskip 	

	As an application, we construct an embedded resolution of singularities for generic determinantal varieties $ X_{m,n,r} $ which are defined  by the $ r \times r $ minors of a generic $ m \times n $ matrix (i.e., whose entries are independent variables), $ r \leq m \leq n \in \IZ_+ $.  
	This is a slight generalization of a result by Vainsencher \cite{DetermResol} to the case of non-algebraically closed fields and even mixed characteristic. 
	More precisely, we prove

\begin{theorem}[Theorem \ref{Thm:ResDetText}]
	\label{Thm:ResDet}
	Let $ m , n, r \in \IZ_+ $ with $ r \leq m \leq n $, let $ R_0 $ be a regular ring, and let $ M = (x_{i,j})_{i,j} $ be a generic $ m \times n $ matrix.
	The following sequence of blowing ups provides an embedded resolution of singularities for the generic determinantal variety $ X_{m,n,r} \subset Z = \Spec(R_0[x_{i,j}\,|\,i,j\,]) $,
	$$ 
	Z =: Z_0 \stackrel{\pi_1}{\longleftarrow} Z_1
	\stackrel{\pi_2}{\longleftarrow} 
	\ldots 
	\stackrel{\pi_{r-1}}{\longleftarrow} Z_{r-1}, 
	$$  
	where $ \pi_\ell $ is the blowing up with center the strict transform of $ X_{m,n,\ell} $ in $ Z_{\ell-1} $, $ 1 \leq \ell \leq r - 1 $.
\end{theorem}

After recalling in sections 1 and 2 the language of idealistic exponents, some notions, invariants and results, we make the reader familiar with the techniques and ideas for reduction by considering several examples in section 3. 
In section 4, we prove Theorems \ref{Thm:EquivIntro} and \ref{Thm:ReductionIntro}, 
and discuss the characterization procedure in detail.
In the last section we show Theorem \ref{Thm:ResDet}.

\smallskip 

\emph{Acknowledgement:} 
The author heartily thanks Bernhard Dietel for many discussions and explanations on his thesis.
He is grateful to Edward Bierstone, Anne Fr\"uhbis-Kr\"uger, and Ulrich G\"ortz for many useful comments on an earlier draft of the paper.
Further, he thanks the University of Mainz and the University of Versailles for their support and excellent working environment during his time there, when major parts of the article were developed.

%
%
%
%
%
%
%
%		Idealisitc exponents
%
%
%
%
%
%
%

\medskip

\section{Idealistic exponents in a nutshell}

We give a brief overview on the language of idealistic exponents that we are going to use.
(For more details we refer to \cite{HiroIdExp, HiroThreeKey, BerndIdExp}).
Nowadays, there are several variants (e.g.~marked ideal \cite{BMfunct, Kollar, Jarek}, basic objects \cite{AnaSantiOrlando05}, presentations \cite{BMheavy}, a variant in the terms of Rees algebras \cite{AngelicaOrlandoElimination, AnaOrlandoSimpli}, or idealistic filtrations \cite{IFP1, IFP2, HirakuKenjiObergurgl}) each different in technical details but all having their origins in Hironaka's notions of pairs and idealistic exponents \cite{HiroIdExp}.

\begin{Rk} 
	Since the theory of idealistic exponents is a bit technical, let us first motivate it by discussing the hypersurface case:
	Consider  $ X = V (f) \subset \Spec (R) = Z $, where $ R $ is a regular local ring and $ 0 \neq f \in R $.
A rough measure for the complexity of the singularity of $ V(f) $ at a point $ x \in \Spec(R) $ is the order of $ f $, 
$$ 
\ord_P(f) := \sup \{ \, d \in \IZ_{ \geq 0 } \mid f \in P^d \, \} \,,
$$ 
where $ P = P_x $ is the prime ideal defining the point $ x \in \Spec ( R ) $.
The worst points \wrt the order are those with order equal to
$$ 
b := \max \{ \ord_{P_x} ( f ) \mid x \in \Spec ( R ) \}
.$$
Hence we memorize the {\em pair} $ \IE := ( f, b ) $ and we denote the set of points where the order of $ f $ is maximal by $ \Sing ( f, b ) $.
Note that $ \Sing(f,b) $ is closed by the upper semi-continuity  of the order
(see \cite{Hiro64} Chapter III \S 3 Corollary 1, p.220).

The goal is to find a sequence of finitely many permissible blowing ups (Definition \ref{Def:perm}) such that, for every point $ x' $ lying on the strict transform $ V(f') $ of $ V(f) $, we have $ \ord_{P_{x'}} ( f' ) < b $. 
In other words, $ \Sing ( f', b ) = \varnothing $ which, by definition, means that the pair $ ( f', b ) $ is resolved.
It is important to mention that the variety $ V(f') $ is not necessarily resolved (see Examples \ref{Ex:easy} and \ref{Ex:Not_full_Reso}). 
But we attained a strict improvement of the singularity since the order can not increase under permissible blowing ups.

We start over, i.e., we consider $ ( f', b') $ with $ b' < b $ the new maximal value of the order of $ f' $.
Assuming that all this is possible we reach the case $ \Sing ( f'', 2 ) = \varnothing $, where $ V(f'') $ denotes the strict transform of $ V(f) $.
Since there is no point of order two or bigger the original singularity is resolved.
This reduces the problem of finding a resolution of singularities for $ X $, to the task of constructing a resolution for $ \IE $.

In order to achieve the desired sequence of permissible blowing ups one has to overcome technical difficulties and refine the previous principle by considering the information on the preceding blowing ups which is encoded in the exceptional divisors. 
Note also that we did not address condition (3) of Conjecture \ref{Conj:ERS}, yet. 
Further, one may have to discuss the question
if locally constructed resolutions patch together to a global one.
Therefore it is preferable to avoid any non-intrinsic choice when determining the center for the blowing ups. 
For an example on this issue and how it is overcome, see Example \ref{Ex:easy_B} or \cite{CS_Dim2} Example 1.14.

From the viewpoint of resolving singularities it is reasonable to identify pairs which undergo the same resolution process.
This provides an equivalence relation on the set of pairs on $ R $, for which an equivalence class is called {\em idealistic exponent}. 
\end{Rk}

Let us come to the precise definitions.
Although idealistic exponents are defined in a more general setting, we restrict ourselves to the local situation.

\begin{Def}
\label{Def:idexp}
Let $ R $ be a regular local ring.
	A \emph{pair} $ \IE = (J,b) $ on $ R $ is a pair consisting of an ideal $ J \subset R $ and a positive integer $ b \in \IZ_+ $.
	We define its \textit{order} at a point $ x \in \Spec(R) $ (corresponding to a prime ideal $ P = P_x \subset R $)
	as 
	$$ 
		\ord_x (\IE) := 
		\left\{ \begin{tabular}{cl}
			$ \frac{\ord_P (J) }{b} $ 	& , if $ \ord_P (J)  \geq b $ and \\[8pt]
			$ 0 $				& , else.  %%%,
		\end{tabular} \right.
	$$ 
	We define the \textit{singular locus} (or \textit{(co)support}) of $ \IE $ as 
	$$
		\Sing(\IE) := \{ x \in \Spec(R) \mid \ord_{P_x} ( J ) \geq b \}.
	$$
	If $ \Sing ( \IE ) = \varnothing $ then $ \IE $ is called \textit{resolved}.
	
	Let $ \IE_1 := (J_1, b_1) $ and $ \IE_2 := ( J_2, b_2 ) $ be two pairs  on $ R $.
	Set $ c := \mathrm{lcm}(b_1,b_2) $ and $ c_i = \frac{c}{b_i} $, for $ i \in \{ 1, 2 \} $.
	We define the \textit{intersection of $ \IE_1 $ and $ \IE_2 $} 
	to be the pair
	\begin{equation} 
	\label{eq:intersect_pairs} 
		 \IE_1 \cap  \IE_2 := ( J_1^{c_1} + J_2^{c_2}, c ).
	\end{equation}
	Similarly, we define the intersection of arbitrarily finitely many pairs.
\end{Def}

Since $ R $ is regular it is Noetherian, by definition.
Here, we do neither impose on $ R $ to contain a field, nor to fulfill condition \eqref{condD}.
Again, $ \Sing (\IE) $ is closed since the order is upper-semi continuous.
It is not hard to show that we have $ \Sing ( \IE_1 \cap \IE_2 ) = \Sing ( \IE_1 ) \cap \Sing ( \IE_2 ) $
and $ \ord_x ( \IE_1 \cap \IE_2 ) = \min  \{  \ord_x ( \IE_1 ),  \ord_x ( \IE_2 ) \} $, for  $ x \in \Sing ( \IE_1 \cap \IE_2 ) $.

\begin{Def}
	\label{Def:standardbasis(2)}
	Let $ R $ be a regular local ring with maximal ideal $ M $ and $ J_0 \subset R $ be a non-zero ideal in $ R $.
	\begin{enumerate}
		\item 
	The {\em associated graded ring of $ R $} is defined by 
	$$
		gr_M(R) := \bigoplus_{i \geq 0} M^i/M^{i+1}.
	$$
	For a non-zero element $ g \in R $, its {\em initial form \wrt $ M $} is the element $ in_M(g) \in gr_M(R) $  defined by
	$$
		in_M(g) := g \mod M^{b + 1},
		\hspace{15pt} \mbox{ where \ } b := \ord_M(g).
	$$
	We define the {\em initial ideal of $ J_0 $ \wrt $ M $} as the ideal in $ gr_M(R) $ given by
	$$
	In_M(J_0) := \langle in_M (g) \mid g \in J_0 \rangle
	\subset gr_M(R). 
	$$
	Analogously, $ gr_P (R) $, $ in_P(g) $, and $ In_P (J_0) $ are defined for any non-zero prime ideal $ P \subset R $.
	
	\medskip 
	
	\item 
	Let $ ( f) = ( f_1, \ldots, f_m ) $ be a set of elements in $ J_0 $.
	Set $ b_i := \ord_M ( f_i)  $ and $ F_i := in_M (f_i ) $, $ 1 \leq i \leq m $,
	We say, $ ( f ) $ is a {\em standard basis} for $ J_0 $
	(in the sense of Hironaka) if the following holds:
	\begin{enumerate}
		\item	$ \langle F_1, \ldots, F_m \rangle = In_M(J_0) \subset gr_M (R)$,
		\item	$ b_1 \leq b_2 \leq \ldots \leq b_m $, and
		\item	$ F_i \notin \langle F_1, \ldots, F_{i-1} \rangle $, for all $ i \in \{ 2, \ldots, m \} $.
	\end{enumerate}
	\end{enumerate}
\end{Def}

The notion of standard bases is closely related to that of Macaulay bases, but for the latter conditions (b) and (c) above are not required. 

Note that $ gr_M(R) $ is isomorphic to a polynomial ring over the residue field $ k := R/M $.
More precisely, if $ (w) = ( w_1, \ldots, w_n ) $ is a \RSP for $ R $ and 
if we denote by $ W_i := w_i \mod M^2 $ the images in the graded ring, $ 1 \leq i \leq n $, then
$ gr_M ( R ) \cong k[ W_1 , \ldots, W_ n ] $.

By results of Hironaka,
a standard basis for $ J_0 $ generates $ J_0 $
(\cite{HiroCharPoly}, Corollary (2.21.d))
and the strict transforms of $ ( f ) $ under a reasonable blowing up still generate the strict transform of the ideal $ J_0 $ under the same blowing up
(\cite{Hiro64}, III.2, Lemma 6, p.~216, and see also III.6, Theorem 5, p.~238).
In general, this is not true for any set of generators, e.g.~$ I := \langle y^2 - x^3, y^2 - z^5 \rangle \subset R := k[ x,y,z]_{\langle x,y,z \rangle } $ and the blowing up of the maximal ideal, for $ k $ any field.

\begin{Rk}
\label{Rk:IE_to_x}
Let us explain how to associate a pair to a given singularity $ X \subset Z $, where $ X  $ is a closed subscheme of an excellent regular Noetherian scheme $ Z $
(not necessarily containing a field).
Let $ x_0 \in X $ and let $ U = \Spec ( R ) \subset Z $ be a sufficiently small affine neighborhood of $ x_0 $, 
where $ R = \cO_{Z,x_0} $ is the regular local ring whose maximal ideal $ M $ corresponds to $ x_0 $.
The closed subscheme $ X \cap U \subset U $ is given by an ideal $ J_0 \subset R $.
Let $ ( f ) = ( f_1, \ldots, f_m ) $ be a standard basis for $ J_0 $. 
The pair associated to $ J_0 \subset R $ is defined by
$$
	\IE = ( f_1, b_1 ) \cap ( f_2, b_2 ) \cap \ldots \cap ( f_m, b_m ).
$$
(Recall that $ b_i := \ord_M(f_i) $).
If we can resolve $ \IE$, then the Hilbert-Samuel function must drop locally above $ x_0 $
(\cite{CJS} Theorem 2.10 (2), (6) and Definitions 1.26 and 1.28) 
\end{Rk} 

\begin{Def}
\label{Def:perm}
	Let $ \IE = (J, b) $ be a pair on $ R $.
	\begin{enumerate}
		\item A closed subscheme $ D \subset \Spec(R)  $ is called {\em permissible center for $ \IE $} (or {\em permissible for $ \IE $}) if $ D $ is regular and $ D \subseteq \Sing (\IE) $.	
	If $ D $ is permissible for $ \IE $ then we also say that the blowing up $ \pi: Z' \to \Spec(R) $ with center $ D $ is permissible for $ \IE $.
	
	\medskip
	
	\item
	Let $ \pi $ be a permissible blowing up for $ \IE $ and $ U' = \Spec ( R') \subset Z' $ an affine chart.
	The transform of $ \IE $ in $ U' $ is then given by $ \IE' = (J', b) $, where $ J' $ is defined via $ J \cdot R' = J' H^b $, where $ H $ denotes the ideal of the exceptional divisor in $ U' $.
	
	We call a composition of the form $ U' \hookrightarrow Z' \to \Spec(R) $ a {\em local blowing up}.
	A sequence of local blowing ups is said to be permissible for $ \IE $ if each of the blowing ups is permissible for the corresponding transform of $ \IE $.
	\medskip
	
	\item
	A finite sequence of permissible local blowing ups such that the singular locus of the final transform of $ \IE $ is empty (i.e., the final transform is resolved) is called a {\em (local) resolution of singularities for $ \IE $}. 
\end{enumerate}
\end{Def}

\medskip

Note that in other literature there exist notions of permissible centers which are different from the latter, see for example \cite{CJS} Definition 2.1, where additionally $ X $ has to be normally flat along $ D $.

\smallskip

\begin{Ex}
\label{Ex:easy}
	Let $ R = k[[x,y,z]] $, where $ k $ is any field, and let $ J = \langle x^3 - y^3 z^2 \rangle $.
	Then 
	$$
		\begin{array}{ll}
			\Sing ( J , 2 ) = V ( x, y ) \cup V( x, z )\,, \\
			\Sing ( J , 3 ) = V ( x, y )\,.  		  		
		\end{array}
	$$
	Hence $ V ( x, z ) $ is permissible for $ ( J , 2 ) $ but not for $ ( J , 3 ) $, whereas $ V ( x, y ) $ is permissible for both.
	We blow up with center $ V ( x, y ) $ and consider the point with coordinates $(x',y',z') = (\frac{x}{y}, y, z ) $.
	The	total transform of $ J $ is $ J \cdot R' = \langle  y'^3 (x'^3  -  z'^2) \rangle $ and the ideal of the exceptional divisor is 
	$ H = \langle y' \rangle $.
	Hence the transform of $ ( x^3 - y^3 z^2 , 2 )  $ 
	is $ ( y' (x'^3  -  z'^2) , 2 )  $.
	On the other hand, the transform of $ ( x^3 - y^3 z^2 , 3 )  $ at the same point is $ ( x'^3 - z'^2 , 3 )  $ and the singular locus of this pair is empty.
	
	The example also illustrates that the transform of a pair $ \IE = (J,b) $  is in general neither given by the total transform nor by the strict transform. 
	The transform $ \IE' $ lies in between them and is controlled by the number $ b $.
\end{Ex}
	
From the viewpoint of resolving pairs it is natural to identify those pairs which show the same behavior under permissible blowing ups.
Working with these equivalence classes has the advantage that we can change the representative in order to achieve a situation where good properties are revealed. 

Before giving the precise definition we have to introduce the following:
Let $ \IE = ( J, b ) $ be a pair on $ R $
and
let $ ( t ) = ( t_1, \ldots, t_a ) $ be an arbitrary finite system of independent indeterminates.
Then the lift of $ \IE $ to $ R[t] $ is defined as $ \IE[t] = (J \cdot R[t], b) $.

\begin{Def}
	\label{Def:sim}
	Two pairs $ \IE_1 = ( J_1, b_1) $ and $ \IE_2 = ( J_2, b_2 ) $ on $ R $ are defined to be {\em equivalent}, $ \IE_1 \sim \IE_2 $ if the following holds:
	\begin{center}
	\parbox{11cm}{
			Let $ (t) = ( t_1, \ldots, t_a ) $ be an arbitrary finite system of independent indeterminates.
			Then any sequence of local blowing ups (over $R[t]$) which is permissible for $ \IE_1[t] $ is also permissible for $ \IE_2[t] $ and vice versa. 
		}	
	\end{center}
		An {\em idealistic exponent} $ \IE_\sim $ is the equivalence class of a pair $ \IE $. 
		
\end{Def}

Allowing the extension from $ R $ to $ R[t] $ looks at first sight technical and not necessary but it is useful for proving results on idealistic exponents.
%{\bf marked ideals does not use this I think. we should point this out.}
%
In other literature pairs are sometimes also called idealistic exponents (e.g. \cite{HiroThreeKey}) and there is no distinction made between the representative and the corresponding equivalence class.

Observe: 
If $ \IE_3 $ is another pair on $ R $, then 
$ \IE_1 \cap \IE_2 \sim \IE_3 $ means that a sequence of local blowing ups over $ R[t] $ is permissible for $ \IE_3[t] $ if and only if it is permissible for $ \IE_1[t] $ and $ \IE_2 [t] $.

\begin{Thm}[Numerical Exponent Theorem]
	\label{Thm_Num_Exp}
	Let $ \IE_1 $ and $ \IE_2 $ be two equivalent pairs on $ R $.
	For every $ x \in \Spec(R) $, we have
	$$
		\ord_x( \IE_1 ) = \ord_x ( \IE_2 ).
	$$
	In particular, $ \Sing (\IE_1) = \Sing(\IE_2) $.
	Hence the singular locus is an invariant of the idealistic exponent since it does not depend on the choice of the representative.
\end{Thm} 

For the proof, we refer to \cite{HiroThreeKey} Theorem 5.1, or \cite{BerndThesis} Proposition 1.1.10.

\medskip

One of the key techniques in the local study of singularities in characteristic zero is the method of hypersurfaces of maximal contact and, directly related to it, the coefficient ideal.
They allow to reduce the resolution problem locally to one in lower dimension and thus to apply an induction argument.
The idea of coefficient ideals goes back to Hironaka (in the context of idealistic exponents this appears in \cite{HiroKorean} Theorem 1.3, p.~908, and \cite{HiroIdExp} section 8, Theorem 5, p.~111) and was later developed by Villamayor (for basic objects) and Bierstone and Milman (for presentations).

The counterpart of the coefficient ideal in the setting of pairs is the so called coefficient pair \wrt some regular subvariety $ V ( y_1, \ldots, y_r  ) \subset \Spec ( R ) $,
 where $ ( y )$ is part of a system of parameters
 (see also section 3 of \cite{BerndIdExp}).
 In Remark \ref{Rk:Expl_coeff_pairs_max_cont}, we explain more on its role in resolution of singularities.  
 Before we recall the definition, let us mention the following.
 
\begin{Rk} 
Let $ (u) = (u_1, \ldots, u_e) $ be a system of elements in $ R $ extending $ (y) = ( y_1, \ldots, y_r ) $ to a regular system of parameters.
Let $ b \in \IZ_+ $.
Every element $ f \in R $ has an expansion of the form
\begin{equation}
\label{eq:expansion}
f = f (u,y) = \sum_{\substack{B\in \IZ^r_{\geq 0}\\|B|<b}} f_B(u) y^B + h,
\end{equation}
where $ h \in \langle y \rangle^b  $
and the coefficients
$ f_B(u) \in  R $
do not depend on $ ( y ) $:
We define $ f^{(1)}_0 := f|_{y_1=0} \in R $
and put $ g^{(1)}_0 := y_1^{-1} (f -  f^{(1)}_0)  \in R $.
We repeat, $ f^{(1)}_{i} := g^{(1)}_{i-1}|_{y_1=0} \in R $ and $ g^{(1)}_{i} := y_1^{-1}(g^{(1)}_{i-1} - f^{(1)}_i) $, for $ i \in \{ 1, \ldots , b-1 \} $, and we obtain that
$$
	f = \sum_{i_1= 0}^{b-1} f^{(1)}_{i_1} y_1^{i_1} + h_1,
	\ \ 
	\mbox{ for some }
	\
	h_1 \in \langle y_1 \rangle^b.
$$
If $ r = 1 $, we have the desired expansion.
If $ r > 1 $, we apply induction on $ r $ for each $ f^{(1)}_{i_1} $, and get that
$$
 f^{(1)}_{i_1} = \sum_{\substack{(i_2, \ldots, i_r)\\
 		i_1 + \ldots + i_r < b - i_1}}
 f_{i_2,\ldots,i_r}(u)y_2^{i_2}\cdots y_r^{i_r} + h^{(1)}_{i_1},
$$
where $ h^{(1)}_{i_1} \in \langle y \rangle^{b - i_1} $ 
and $ f_{i_2,\ldots,i_r}(u) \in R $ do not depend on $ ( y ) $.
This provides \eqref{eq:expansion}.
\end{Rk}

\begin{Def}
\label{Def:IdCoeffExp}
	Let $ \IE = (J,b) $ be a pair on $ R $.
	Let $ (u, y) = (u_1, \ldots, u_e, y_1, \ldots, y_r ) $ be a \RSP for the regular local ring $ R $. 
	We define the {\em coefficient pair $ \ID (\IE; u; y )  $ of $ \IE $ \wrt $ (y) $} as the pair, which is given by the following construction:
	For $ f \in J $, we consider an expansion	
	$
		f = \sum_{B} f_B (u) y^B + h,
	$
	as in \eqref{eq:expansion}.
	We set 
	$$ 
		\ID(f;u;y)  := \bigcap\limits_{\substack{B \in\IZ^r_{ \geq 0 } \\[3pt] |B| < b}} ( f_B (u) , \, b -|B| ),
	$$
	and
	$$
		\ID (\IE; u; y) := \bigcap_{ f \in J } \ID(f; u; y) \,.  
	$$
	Note that we may consider $ \ID(\IE;u;y) $ as a pair on $ R $ as well as one on $ {R} / \langle y \rangle $.
\end{Def}

\begin{Facts}
\label{Facts:Id_Exponents}
	Let $ R $ be a regular local ring, let $ \IE_1 = ( J, b ) = ( J_1, b_1 )  $ and $ \IE_2  = (J_2, b_2 )$ be pairs on $ R $, and let $ ( u, y ) $ be a \RSP for $ R $.
	\begin{enumerate}
		\item	For every $ a \in \IZ_+ $, we have that $ (J^a, a b) \sim (J,b) $.
				Moreover,  if $ b_2 = b_1 = b $ then we have that $ (J_1,b) \cap (J_2,b) \sim \left(J_1 + J_2,  b \right) $.
			
		\medskip
		
		\item[(2)]	If $ \Sing(J_1,b_1 +1) = \Sing(J_2,b_2+1) = \varnothing $, then $ (J_1,b_1) \cap (J_2, b_2) \sim (J_1 J_2, b_1 + b_2 ) $.

		\medskip
		
		\item[(3)]	 
		Let $ \cD $ be a differential operator of order $ m  \in \IZ_{\geq 0} $, $ m < b $, on $ R $ such that $ \cD(R) \subset R $.
				We have that $ (J, b) \sim  (\cD J, b - m) \cap (J,b) $.
	
		\medskip
		
		\item[(4)]	If $ \IE_1 \sim \IE_2 $, then $ \ID (\IE_1 ; u; y) \sim \ID (\IE_2; u; y) $.

		\medskip
		
		\item[(5)]	
		We have that $ (y,1) \cap \IE  \sim (y,1)\cap \ID (\IE; u; y ) $.
	\end{enumerate}
\end{Facts}

\begin{proof}[``Proof"] 
	We just give the references for the proofs.
	Part (1) and (2) follow easily from the definitions.
	The arguments for the proof can also be found in \cite{BerndThesis}, Lemma 1.1.8.
	The third part is the so called Diff Theorem, see \cite{HiroThreeKey}, Theorem 3.4, or \cite{BerndIdExp}, Proposition 1.9.
	The remaining part, (4) resp.~(5), is proven in \cite{BerndIdExp}, Theorem 3.2 resp.~Corollary 3.3.
	The idea is to assume that the equivalence does not hold and then deduce a contradiction from the existence of a sequence of blowing ups that is permissible for one, but not the other. 
\end{proof}

By (1) we can extend the notions of pairs and idealistic exponents to $ b \in \IQ_+ $;
$ (J, b) $ is defined to be the pair which is equivalent to $ ( J^a, a b ) $, where $ a \in \IZ_+ $ is such that $ a \cdot b \in \IZ_+ $.
%Further, it follows that any idealistic exponents has a representative of the form $ ( J, b ) $ for some ideal $ J \subset R $ and $ b \in \IZ_+ $. 
%
Part (3) and (5) are extremely useful for simplifying the representative of an idealistic exponent.
And (4) is telling us that the coefficient pair $ \ID (\IE ; u; y) $ is an invariant of the idealistic exponent $ \IE_\sim $ and $ (u,y) $. 
Thus it is a reasonable object to consider in this theory. 

In particular, (4) and the Numerical Exponent Theorem (Theorem \ref{Thm_Num_Exp}) imply that 
$$
\delta_x (\IE , u, y) := \ord_x ( \ID (\IE ; u; y ) ) ,
$$
for $ x \in \Spec ( R ) $, is an invariant of the idealistic exponent $ \IE_\sim $.
		In \cite{BerndBM}, the author showed that it is possible to relate this number with certain polyhedra attached to idealistic exponents and that $ \delta_x (\IE , u, y) $ is one of the key ingredients of the invariant of Bierstone and Milman for constructive resolution of singularities in characteristic zero, \cite{BMheavy, BMlight}.

\medskip

\begin{Def}
\label{Def:max_contact}
	Let $ R $ be a regular local ring and let $ \IE $ be a pair on $  R $.
	Let $ (y) = (y_1, \ldots, y_r) $ be a system of elements in $ R $ which can be extended to a {\RSP} for $ R $.
	We say $ W := V(y) $ has {\em maximal contact with $ \IE $} if the following equivalence holds
	$$ 
		\IE \sim (y, 1) \cap  \IE \, .
	$$
\end{Def}

Facts \ref{Facts:Id_Exponents} (1) and (3) provide that a system $ ( y) $ having maximal contact can be determined by using differential operators.
For concrete examples see the section 3.

\begin{Rk}
\label{Rk:Expl_coeff_pairs_max_cont}
	Let $ \IE $, $ R $ and $ (y) $ be as in the previous definition. 
	Suppose \eqref{condD} holds for $ R $ and $ V( y) $ has maximal contact with $ \IE $.
	The equivalence implies that $ \Sing ( \IE ) \subset  V ( y ) $, i.e., any center which is permissible for $ \IE $ is of the form $ V ( y, g_1, \ldots, g_d) $ for some $ g_1, \ldots g_d $ coming from $ R/\langle y \rangle $.
	Facts \ref{Facts:Id_Exponents}(4) provides that $ \IE \sim (y,1)\cap \ID (\IE; u; y ) $.
	It is not hard to see that this condition is stable under permissible blowing ups as long as the transform of $ \IE $ is not resolved.
	Furthermore, if we consider one of the $ Y_j $-charts of the blowing up, the transform of $ \IE $ is resolved, since so is $ ( y, 1 ) $.
	
	Unfortunately, maximal contact does not always exist in general. 	
	This problem is discussed in the next sections.
\end{Rk}

\begin{Ex}
\label{Ex:Explain_basics}
	Let $ R = k[[x,y,z]] $, where $ k $ is a field of characteristic zero or $ p > 3 $.
	Consider the hypersurface given by the polynomial
	$$ 
		f = x^2 + y^3 + 3 y^2 z + 3 y z^2 + z^3 + z^{5} \,.
	$$
	We associated the following pair on $ R $ to this:
	$$ 
		\IE:= ( f, 2 ) \sim ( x, 1 ) \cap ( y^3 + 3 y^2 z + 3 y z^2 + z^3 + z^{5} , 2 ) 
	$$
	The equivalence follows from Facts \ref{Facts:Id_Exponents}(3) by taking the derivative by $ x $ and using that $ ( 2x, 1 ) = ( x,1 ) $ since $ 2 $ is invertible in $ k $.
	
	The coefficient pair is	$ \ID_1 :=  \ID (\IE ; y,z; x ) = ( y^3 +  3 y^2 z + 3 y z^2 + z^3 + z^{5} , 2 ) $ and $ \delta_1 := \delta (\IE , u, y ) = \frac{3}{2} $. 
	By the above equivalence, we can deduce a resolution of $ \IE $ from a resolution of $ \ID_1 $. 
	For the latter, we first want to eliminate the singularities of highest order, i.e., in this case those of order $ 3 $.
	Hence we consider
	$$ 
		\ID_1^* := ( y^3 +  3 y^2 z + 3 y z^2 + z^3 + z^{5} , 3 ) 
	$$
	(Note that $ 3 = \delta_1 \cdot 2 $).
	After resolving
	$ 
		\ID_1^* ,
	$ 
	we have to study the transform $ \ID'_1 $ of $ \ID_1 $ under the preceding blowing ups.
	This is slightly more involved and we refer the reader to the next section for more details.

	Using Facts \ref{Facts:Id_Exponents}(3), we get
	$ \ID_1^* \sim ( 6 y + 6 z , 1  ) \cap \ID_1^* $ after taking two times the derivative by $ y $.
	We set	
	$ w := y + z $ and replace $ y $ by $ w - z $.
	(Note that $ 6  $ is invertible in $ k $).
	We have
	$ ( y + z)^3 + z^{5} = w^3 + z^{5} $
	and
	$
	 \ID_1^* \sim ( w ,1 ) \cap ( z^{5}, 3)
	$.
	Therefore the next coefficient pair is
	$ \ID_2 := ( z^{5}, 3) $
	and
	$ \delta_2 := \frac{5}{3} $.
	We modify the assigned number as before and get	
	$ \ID_2^* = ( z^{5}, 5 ) \sim ( z, 1 ) $.
	
	Since $ f = x^2 + w^3 + z^5 $, the origin $ V ( x,w,z) = V ( x,y,z) $ is the only possible permissible center for the first blow-up.
	We leave it as an exercise to the reader to compute the transforms of the considered pairs under the blowing up.
\end{Ex}

\begin{Rk}
\label{Rk:Blowup_on_V(y)_back_to}
	Suppose we are in the situation of maximal contact, $ \IE \sim ( y, 1) \cap \ID (\IE; u; y ) $ and assume by induction  $ \ID (\IE; u; y ) $ yields the center $ D_Y := V ( \overline g_1, \ldots, \overline g_d ) $, for certain $ \overline g_i \in R/ \langle y \rangle $. 
	Then the center for $ \IE $ is given by $ D := V ( y, g_1, \ldots, g_d ) $, where $ g_i \in R $ denotes any lift of $ \overline g_i $.
	In Example \ref{Ex:More_detail_Example} we explain more in detail how this improves the original singularity. 
\end{Rk}

\smallskip

For our characterization result we need some technical notions.
Namely, the idealistic variant of the tangent cone and certain objects associated with it (the so called directrix and the ridge, which were introduced and studied by Hironaka and Giraud. 
These are invariants of the idealistic exponent that provide some refined information on it.
For a detailed discussion we refer to section 2 of \cite{BerndIdExp}.

\begin{Def}
\label{Def:Tangent}
	Let $ \IE = (J, b) $ be a pair on $ R $, $ M $ be the maximal ideal of $ R $, and $ k = R/M $ be its residue field.
	Suppose $ \Sing (\IE) \neq \varnothing $.
	\begin{enumerate}
		\item	For $ f \in J $, the {\em $ b $-initial form of $ f $ (at $ M $)} is defined by 	
				$$
					in (f,b) := \left\{ \hspace{5pt} \begin{tabular}{cc}
									$ f \mod M^{ b + 1 } $, 	&	 if $ b \in \IZ_+ $, \\[5pt]
									$ 0 $,					&	 if $ b \notin \IZ_+ $.
									\end{tabular}
							\right.		
				$$
		
		\medskip
		
		\item	We define the {\em tangent cone 
			$ C_M(J,b) \subset \Spec (gr_M(R)) $ of $ \IE $ at $ M $} as the cone defined by the homogeneous ideal $ In_M (J,b) \subset  gr_M (R) $, where
				$$
				In_M (J,b) :=  \left\{ \begin{tabular}{cc}
									$ \langle  J \mod M^{ b + 1 } \rangle = \langle in (f, b) \mid f \in J \rangle $, 	&	 if $ b \in \IZ_+ $, \\[5pt] 
									$ \langle 0 \rangle $,					&	 if $ b \notin \IZ_+ $.
									\end{tabular}
							\right.		
			$$
	\end{enumerate} 
	
\end{Def}

\begin{Rk}
		The condition $ \Sing( \IE) \neq \varnothing $ is equivalent to $ x_0 \in \Sing ( \IE) $, 
		where $ x_0 $ denotes the closed point corresponding to the maximal ideal $ M $,
		or, in other words, it is equivalent to $ \ord_M(J) \geq b $.
		Therefore $ \Sing (\IE) \neq \varnothing $ implies $ f \in M^b $, for every $ f \in J $.
		Hence, $ in(f,b) $ is either homogeneous of degree $ b $ (if $ \ord_M(f) = b $), 
		or zero (if $ \ord_M(f) > b $).
		
		The latter also justifies to put $ in(f,b) := 0 $ if $ b \in \IQ_+ \setminus \IZ_+ $.
		Since $ \ord_M (J) \in \IZ_\gqz $ and $ \ord_M(J) \geq b $, we have $ \ord_M(J) > b $ in this case and thus $ \ord_M(f) > b $ for every $ f \in J $.
\end{Rk}

\begin{Ex}
	\label{Ex:continues}
	Let $ R $ be a regular local ring with maximal ideal $ M $ and residue field $ k = R/M $, $ \car{k} = p > 0 $. 
	Let $ (x,y,z, t,u,v) $ be a regular system of parameters for $ R $.
	We have that $ gr_M(R) \cong k [X,Y,Z,T,U,V] $.
	\begin{enumerate}
		\item 
	Consider the pair $ \IE = (J,b) := ( \langle f_1, f_2, f_3 \rangle, p^2 + 1 ) $, for 
	$$ 
	f_1 := x y^{p^2}  - x t^3 u^{p^2},
	 \ \ \ \
	f_2 := z^{p^2 - p + 1} (t + u)^{p} - v^{p^3},
	 \ \ \ \
	 f_3 := t^{p^2 + 2} - u^{p^2 + 1} v.
	 $$
	 By the above definition:
	 $
	 in(f_1, p^2 + 1) 
	 = 	X Y^{p^2}, 
	  \ \
	 in(f_2, p^2 + 1) 
	 = 	Z^{p^2 - p + 1} ( T + U)^{p},
	 $ 
	 and 
	 $
	 in(f_3, p^2 + 1) 
	 = 	0.
	 $
	 Therefore, $ In_M( J,b ) = \langle XY^{p^2}, \, Z^{p^2 - p + 1} ( T + U)^{p} \rangle \subset gr_M(R)$. 
	 
	 \medskip  
	 
	 \item 
	 Suppose that $ k $ is not perfect and let $ \lambda \in k \setminus k^p $.
	 Let $ \epsilon \in R^\times $ be a unit in $ R $ such that $ \epsilon \equiv \lambda \mod M $.
	 Let $ \IE' = (J', b') := ( \langle f'_1, f'_2 \rangle , p^2 ) $ be the pair defined by
	 $$
		 f'_1 := (x^p  + \epsilon y^p )z^{p^2 - p} + tuv^{p^2},
		 \ \ \ \ 
		 f'_2 := z^{p^2} + u^{p^2} + \epsilon ( x^p  + \epsilon y^p  )^p + v^{p^2 + 1}.
	 $$
	 We get: $ in(f'_1,p^2 ) = (X^p  + \lambda Y^p )Z^{p^2 - p } $ and
	 $ in(f'_2,p^2 ) = Z^{p^2} + U^{p^2} + \lambda ( X^{p}  + \lambda Y^p )^p $,
	 and $ In_M( J',b' ) $ is the ideal in $ gr_M(R)$ generated by these two. 
 	\end{enumerate}
\end{Ex}

\begin{Def}
	\label{Def:Dir_Rid}
	Let $ \IE = (J, b) $ be a pair on $ R $, $ M $ be the maximal ideal of $ R $, and $ k = R/M $ be its residue field.
	Suppose $ \Sing (\IE) \neq \varnothing $.
	\begin{enumerate}		
		\item	The {\em directrix $ \Dir (\IE ) = \Dir_{M} (\IE )$ of $ \IE $} is defined to be 
		the largest sub-vector space $ V $ of $ \Spec(gr_M(R)) $ leaving the cone $ C_M(J,b) $ stable under translation, $ C_M(J,b) + V  = C_M(J,b) $.
		In other words, $ \Dir(\IE) $ corresponds to a smallest list of variables 
		$ Y_1,\ldots Y_r \in  gr_M ( R )_1 $ (homogeneous of degree one) such that
		\begin{equation}
			\label{eq:directrix}
				(\, In_M (J,b) \,\cap\, k[Y_1,\ldots,Y_r ] \,)\cdot gr_M ( R ) = In_M (J,b) .
		\end{equation}
			
		\medskip

		\item	The {\em ridge  (fa\^ite in French) $  \Rid (\IE ) = \Rid_{M} (\IE )$ of $ \IE $} is defined to be 
		the largest additive subgroup of $ \Spec(gr_M(R)) $ leaving the cone $ C_M(J,b) $ stable under translation.
		In other words, $ \Rid(\IE) $ corresponds to a smallest list of additive homogeneous polynomials $  \sigma_1, \ldots, \sigma_s  \in  gr_M ( R ) $ such that
			\begin{equation}
			\label{eq:ridge}
				(\, In_M (J,b) \,\cap\, k[ \sigma_1, \ldots, \sigma_s ] \,)\cdot gr_M ( R ) = In_M (J,b) .
			\end{equation}
	\end{enumerate} 
	
\end{Def}

Recall that a polynomial $ \sigma ( W ) \in k[W_1, \ldots, W_n]  $ is called additive if $ \sigma ( a + b ) = \sigma ( a ) + \sigma ( b ) $, for any $ a, b \in k^n $.
If $ \car{ k } = 0 $ then the additive polynomials are those homogeneous of degree one, i.e., the definition of the directrix and the ridge coincide. 
If $ \car{ k } = p > 0 $ then the additive homogeneous polynomials are of the form $ \sigma = \sum_{i = 1 }^n \lambda_i W_i^q $, $ \lambda_i \in k $ and $ q = p^d $, $ d \in \IZ_{ \geq 0 } $.

\begin{Ex}
	\label{Ex:continues2}
	Let us continue Example \ref{Ex:continues}.
	\begin{enumerate}
		\item 
	We have 
	$
		In_M( J,b ) = \langle XY^{p^2}, \, Z^{p^2 - p + 1} ( T + U)^{p} \rangle \subset gr_M(R).
	$
	At the first look, one might think that the variables 
	$ (Y_1, \ldots, Y_r) $ corresponding to the directrix $ \Dir(\IE) $ are $ (X, Y, Z, T , U) $.
	But we may replace $ t $ by the new parameter $ \widetilde t  := t + u $.
	Then $ Z^{p^2 - p + 1} ( T + U)^{p} 
	= Z^{p^2 - p + 1}  \widetilde T^{p} $ and we see that
	$$
	(Y_1, \ldots, Y_4) = (X,Y,Z,\widetilde T).
	$$
	Moreover, a possible choice for the additive polynomial defining the ridge $ \Rid(\IE) $ is
	$$
		\sigma_1 := X,
		\ \
		\sigma_2 := Z,
		\ \
		\sigma_3 := \widetilde{T}^p,
		\ \
		\sigma_4 := Y^{p^2},
	$$
	and $ XY^{p^2} = \sigma_1 \sigma_4 $ and 
	$ Z^{p^2 - p + 1}  \widetilde T^{p}
	= \sigma_2^{p^2 - p + 1}  \sigma_3 $.
	Note that $ p^2 - p + 1 $ is prime to $ p $.
	
	\medskip 
	
	\item 
	The ideal of the tangent cone $ C_M( J',b' ) $ is generated by
	$$ 
		F_1' := (X^p  + \lambda Y^p )Z^{p^2 - p } 
	\ \ \mbox{ and } \ \
		F_2' := Z^{p^2} + U^{p^2} + \lambda ( X^{p}  + \lambda Y^p )^p .
	$$
	Since $ \lambda $ is not a $ p $-power in $ k $, we can not perform a change in the parameters to reduce the number of variables appearing in the generators as before.
	Therefore, the directrix $ \Dir (\IE') $ is given by
	$$
		(Y_1', Y_2', Y_3', Y_4')
		= (X, Y, Z, U).
	$$
	(We use $ Y_i' $ in order to avoid confusion with the first example).
	Furthermore, a possible choice for the additive polynomials giving $ \Rid (\IE') $ is
	$$
		\sigma_1' = X^p + \lambda Y^p,
		\ \
		\sigma_2' = Z^p,
		\ \
		\sigma_3' = F_2'.
	$$
	(Again, we use $ \sigma_j' $ to distinguish this from the previous example).
	We have that $ F'_1 = \sigma_1' (\sigma_2')^{p-1} $ and $ F'_2 = \sigma_3' $.
\end{enumerate}
\end{Ex}   

\begin{Def}
	\label{Def:(3)_defines directrix_ridge}
	Let the situation be as in the previous definition.
\begin{enumerate}
\item
We say that additive polynomials $ \sg_1, \sg_2, \ldots, \sg_s  \in K [ Y_1, \ldots, Y_r ] $ are in {\em triangular form} if
\begin{equation}
\label{eq:triangular}
	\sg_i = Y_i^{ q_i } + \sum_{ j = i + 1 }^{ r }  \lm_{ij} \, Y_{ j }^{ q_i },
\end{equation}
where $ q_i = p^{ d_i } $, for certain $ d_i \in \IZ_{\geq 0 } $, and $ q_i \leq q_{ i +1 } $, and $ \lm_{ij} \in k $, $ 1 \leq i \leq s $.

Note that, for every set of additive polynomials, there exists a set of additive polynomials in triangular form which generates the same ideal.
One could even achieve that $ \lm_{ij} = 0 $ for every $ j $ with $ q_i = q_j $.
But for our purpose it is only important that $ ( \sg ) = ( \sg_1, \ldots, \sg_s ) $ are in triangular form.

\medskip

\item
If we can choose the system $ ( Y ) = ( Y_1, \ldots, Y_r ) $ in such a way that $ \sigma_j = Y_j^{q_j} $, for all $ j $, then we say that {\em the reduced ridge coincides with the directrix}.
(Note that then $ s = r $).
For example, this is the case if $ k $ is perfect.
In general, this is only true after a finite pure-inseparable base field extension $ K / k $;
for more on this see \cite{BerndIdExp}, Remark 2.6.
(For more details on the ridge (and in particular an intrinsic definition) see \cite{GiraudEtude} and \cite{BHM}; 
further in \cite{CPS} its role as a refined invariant for a singularity is discussed).

\medskip

\item
We say a system $ ( y )  = ( y_1, \ldots, y_r ) $ (resp. $ (g) = ( g_1, \ldots, g_s ) $) in $ R $ {\em determines (or defines) the directrix (resp.~the ridge)} if their initial forms $ ( Y ) $, $ Y_j = y_j \mod M^2 $, $ 1 \leq j \leq r $, (resp. $ (\sigma) $, $ \sigma_i = g_i \mod M^{q_i + 1} $, $ 1 \leq i \leq s $) fulfill \eqref{eq:directrix} (resp.~\eqref{eq:ridge}).
Here we always implicitly assume that $ r $ (resp.~$ s $) is minimal with this property.
In the same way, we say $ ( Y ) $ (resp.$ (\sigma ) $) defines $ \Dir_M(\IE) $ (resp.~$ \Rid_M(\IE) $).
\end{enumerate}
\end{Def}

\begin{Ex}
	\label{Ex:continues3}
	Let us continue the previous examples.
	\begin{enumerate}
		\item 
		For $ \IE $, the additive polynomials $ (\sigma_1, \ldots, \sigma_4 ) $ 
		are in triangular form and the reduced ridge coincides with the directrix. 
		Further, $ g_1 := x, g_2 := z ,  g_3 := \widetilde t^p $, and $ g_4:=y^{p^2} - t^3 u^{p^2}$ is a possible choice for elements in $ R $ that determine the ridge of $ \IE $.
		
		\medskip
		
		\item 
		For $ \IE' $,
		the additive polynomials $ (\sigma_1', \sigma_2', \sigma_3' ) $ are not in triangular form. 
		In order to achieve the latter, we need to replace $ \sigma_3' = Z^{p^2} + U^{p^2} + \lambda (X^p + \lambda Y^p )^p $ by 
		$ \sigma_3^* := U^{p^2} $.
		Note that $ (\sigma_1', \sigma_2', \sigma_3^* ) $ generates the same ideal as $ (\sigma_1', \sigma_2', \sigma_3' ) $.
		Since $ \lambda \notin k^p $ is not a $ p $-power, the polynomial
		$ \sigma_1' = X^p + \lambda Y^p $ is reduced and the reduced ridge does not coincide with the directrix.
		Finally, $ g'_1 := x^p + \epsilon y^p ,
		g'_2 := z^p $, and $ g'_3 := u^{p^2} + v^{p^2 - 1} $
		are elements in $ R $ defining $ \Rid(\IE') $.
	\end{enumerate}
\end{Ex}

\medskip

By passing to the localization $ R_P $ all the previous notions are defined for any point $ x \in \Sing( \IE ) $, $ x $ corresponding to $ P \subset R $  a non-zero prime.
Further, all definitions and result can be applied for $ R $ being a polynomial ring over a field (or any regular local ring) instead of a regular local ring.

\medskip

\begin{Def}
	\label{Def:idealistic_tangent_cone_etc}
	Let $ \IE = (J,b) $ be a pair on $ R $ and suppose $ \Sing ( \IE ) \neq \varnothing $.
	Let $ ( Y ) =  ( Y_1, \ldots, Y_r ) $ (resp. $ ( \sigma ) = ( \sigma_1, \ldots, \sigma_s ) $) be elements defining the directrix (resp.~the ridge) of $ \IE $.
	In particular, $ Y_j $ are homogeneous of degree one, $ 1 \leq j \leq r $, and $ \sigma_i $ are additive homogeneous polynomials of order $ p^{d_i} $, $ d_i \geq 0 $, $ 1 \leq i \leq s $. 	
	Then we define the following pairs on $ gr_M (R) \cong k[ Y, U ] $ (where $ ( U ) = ( U_1, \ldots, U_e ) $ are homogeneous elements of degree one extending $ ( Y ) $):
	\begin{center}
		\begin{tabular}{ll}
			$ \ITC ( \IE ) = (\, In_M ( \IE ),\, b \,) $ & \em idealistic tangent cone of $ \IE $ (at $ M $), \\[5pt]
			$ \IDir ( \IE ) = (\, \langle Y_1, \ldots, Y_r \rangle,\, 1 \,) $  & \em idealistic directrix of $ \IE $ (at $ M $), \\[5pt]
			$ \IRid ( \IE ) = \bigcap\limits_{i = 1 }^s (\, \sigma_i ,\, p^{d_i} \,) $  & \em idealistic ridge of $ \IE $ (at $ M $) .
		\end{tabular} 
	\end{center}	
\end{Def}

\bigskip

Let $ \IE_1 = (J_1,b_1) $, $ \IE_2 = (J_2,b_2) $ be two pairs on $ R $ and assume $ \Sing ( \IE_1 \cap \IE_2 ) \neq \varnothing $. 
Then we have 
$
		In_M (\IE_1 \cap \IE_2) = In_M(\IE_1) + In_M(\IE_2) 
$.
From this we obtain the equalities
$ \ITC ( \IE_1 \cap \IE_2 ) =  \ITC ( \IE_1 ) \cap  \ITC (\IE_2 ) $,
$ \IDir ( \IE_1 \cap \IE_2 ) = \IDir ( \IE_1 ) \cap  \IDir ( \IE_2 ) $, and 
$ \IRid ( \IE_1 \cap \IE_2 ) =  \IRid ( \IE_1 ) \cap  \IRid ( \IE_2 ) $.

\medskip
	
\begin{Prop}[\cite{BerndIdExp}, Proposition 2.14, Corollary 2.12]
\label{Prop:ItcDirRidUnique}
	Let $ \IE = \IE_1 = (J_1, b_1) $ and $ \IE_2 = (J_2, b_2) $ be two equivalent pairs on $ R $, $ \IE_1 \sim \IE_2 $, and suppose $ \Sing (\IE_1) = \Sing (\IE_2) \neq \varnothing $.
	We have 
	\begin{equation}
	\label{eq:ideal_tan_dir_rid_equiv}
		\ITC (\IE_1) \sim  \ITC (\IE_2) \,,
		\hspace{10pt}
		\IDir (\IE_1) = \IDir (\IE_2) \,, 
		\hspace{10pt}
		\mbox{and}
		\hspace{10pt}
		\IRid (\IE_1) \sim \IRid (\IE_2)\,.
	\end{equation}
	Moreover, if $ \car{ k }  = 0 $ or $ b < \car{ k }  $, $ k = R/M $, then 
	\begin{equation}
	\label{eq:for_equiv}
		\IDir ( \IE ) \sim \IRid( \IE ) \sim \ITC(\IE) \, .
	\end{equation}
\end{Prop}	

\medskip

The first part, \eqref{eq:ideal_tan_dir_rid_equiv}, shows that the idealistic variant of the tangent cone, the directrix and the ridge of $ \IE $ are well-defined objects for idealistic exponents, $ \ITC (\IE_\sim ), \IDir (\IE_\sim), \IRid (\IE_\sim) $. 
In fact, the directrix is independent of the choice of the representative for $ \IE_\sim $.

Moreover, \eqref{eq:for_equiv} reveals a connection of these objects which is not seen in their non-idealistic variants.
In the proof the generators for the directrix resp.~the ridge are deduced from those of the initial ideal by applying certain differential operators. 
The restriction on the characteristic is necessary to obtain the equivalence on the left hand side.

In fact, if we translate this equivalence back to $ R $, we obtain elements $ ( y_1, \ldots, y_r ) $ which have maximal contact with $ \IE $.
The failure of the equivalence on the left hand side in general is one of the reasons why the proof of resolution of singularities over fields of characteristic zero can not be immediately extended to positive characteristic.
See also the section 3 for concrete examples.

\medskip

In Theorem \ref{Thm:Equiv} we extend the previous result to a characteristic-free variant which builds the basis for our characterization, Procedure \ref{CharacterizationProcedure}.

\begin{Ex}
	Let us come back to our running examples
	(Examples \ref{Ex:continues}, \ref{Ex:continues2}, \ref{Ex:continues3}).
	\begin{enumerate}
		\item 
		For $ \IE $, we have that
		$$
		\begin{array}{l}
			\ITC(\IE) =
			( \langle XY^{p^2}, \, Z^{p^2 - p + 1} \widetilde T^{p} \rangle , p^2 + 1),
			\\[5pt]
			\IDir ( \IE ) =
			(\langle X,Y,Z,\widetilde T \rangle, 1),
			\\[5pt]
			\IRid( \IE ) =
			( \langle X, Z \rangle , 1 ) 
			\cap 
			( \widetilde{T}^p, p)
			\cap 
			(Y^{p^2}, p^2).
		\end{array} 
		$$
		Facts \ref{Facts:Id_Exponents}(1)
		provides 
		$ \IDir(\IE) \sim  \IRid(\IE) $.
		(Recall the definition of the intersection of pairs \eqref{eq:intersect_pairs}).
		If we apply $ \dell{X} $ to $ \IT (\IE) $ and using Facts \ref{Facts:Id_Exponents}(3), we get that
		$
		\IT (\IE) \sim \IT(\IE) \cap (Y^{p^2}, p^2)
		$.
		We apply $ \dell{Y^{p^2} } $, $ \dell{\widetilde{T}^p} $,
			resp.~$ \dell{ Z^{p^2 - p + 1} } $
			to $ \IT(\IE )$.
			(For their precise definition of these differential operators, we refer to \eqref{eq:HasseSchmdit};
			what we use here is the property $ \dell{Y^{p^2}} X^a Y^b Z^d \widetilde{T}^e = \binom{d}{p^2} X^a Y^b Z^{d-p^2} \widetilde{T}^e $, for any $ a, b, c, d, e \in \IZ_\gqz $, and similarly for the others).
			We obtain that
		$$
		\IT (\IE) \sim \IT(\IE) \cap (Y^{p^2}, p^2)
		\cap (X,1) \cap ( Z^{p^2 - p + 1}, {p^2 - p + 1}) \cap (\widetilde{T}^p, p).
		$$
		By Facts \ref{Facts:Id_Exponents}(1) and (2), we eventually have that $ \IT(\IE) \sim \IRid(\IE) \sim \IDir (\IE) $.
		
		Suppose the differential operators 
		$ \dell{x} $,  $ \dell{y^{p^2} } $, $ \dell{\widetilde{t}^p} $,
		and $ \dell{ z^{p^2 - p + 1} } $
		exist in $ R $. 
		If we apply the same operations as before to $ \IE $, we can deduce (using the notations of Example \ref{Ex:continues3}(1)) that
		$ \IE \sim ( \langle g_1, g_2 \rangle , 1 ) \cap (g_3, p) \cap (g_4, p^2) \cap ( \langle v^{p^3}, f_3 \rangle, p^2 + 1 ) $.
		Note that $ \ord_M (( \langle v^{p^3}, f_3 \rangle, p^2 + 1 ) ) > 1 $
		(cf.~Theorem \ref{Thm:Reduction}).
		
		\medskip
		
		\item 
		For $ \IE' $, we have that
		$$
		\begin{array}{l}
		\ITC(\IE') =
		( \langle (X^p  + \lambda Y^p )Z^{p^2 - p }, \, Z^{p^2} + U^{p^2} + \lambda ( X^{p}  + \lambda Y^p )^p , p^2 ),
		\\[5pt]
		\IDir ( \IE' ) =
		(\langle X,Y,Z,U \rangle, 1),
		\\[5pt]
		\IRid( \IE' ) =
		(  X^p + \lambda Y^p , p ) 
		\cap 
		( Z^p, p)
		\cap 
		( U^{p^2}, p^2).
		\end{array} 
		$$
		We have that 
		$ \IT(\IE') \sim \IRid(\IE') 
		\sim 
		(\langle Z, U \rangle, 1 )
		\cap 
		(  X^p + \lambda Y^p , p ) $,
		but $ \IRid(\IE') \not \sim \IDir (\IE' ) $
		(cf.~Theorem \ref{Thm:Equiv}).
		Using the notations of Example \ref{Ex:continues3}(2), we can deduce under an assumption analogous to the one in (1) that
		$
		\IE' \sim ( \langle g_1', g_2' \rangle , p )
		\cap (g_3', p^2 ) \cap (tuv^{p^2}, p^2)
		$
		with $ \ord_M(tuv^{p^2}, p^2) > 1 $
		(cf.~Theorem \ref{Thm:Reduction}).
		We leave the details as an exercise to the reader.
	\end{enumerate}
	Note that the first part of Proposition \ref{Prop:ItcDirRidUnique} implies $ \IE \not\sim \IE' $ since $ \IDir (\IE) \neq \IDir (\IE') $.
\end{Ex}

\smallskip

To end this section let us mention the following extension of the previous notions which become important if one aims to obtain a reasonable resolution procedure.
In particular, this is useful to achieve condition (3) of Conjecture \ref{Conj:ERS}.

Let $ Z $ be a regular Noetherian scheme.
Recall that a reduced (Cartier) divisor $ B \subset Z $ is called a {\em simple normal crossing divisor} if 
each irreducible component of $ B $ is regular and they intersect transversally.

Let further $ D \subset Z $ be a regular closed subscheme and denote by $ \pi : Z' \to Z $ the blowing up with center $ D $.
We say that $ D $ has {\em at most simple normal crossings with $ B $} if the reduced divisors $ (E' \cup B')_{red} $, 
associated to the union of the exceptional divisors $ E' := \pi^{-1}(D) $ of $ \pi $ and the strict transform $ B' $ of $ B $, is a simple normal crossing divisor. 

\begin{Def}
\label{Def:scnd_perm}
	Let $ \IE = (J, b) $ be a pair on $ R $ and $ D \subset \Spec(R) $ a closed subscheme.
	Let $ \cB $ be a finite collection of irreducible (Cartier) divisors on $ \Spec ( R ) $ and denote by $ B $ the reduced divisor obtained from the union of the elements of $ \cB $.
	Suppose $ B $ is a simple normal crossing divisor.
	
	\begin{enumerate}
	\item
	Then $ D $ is called {\em permissible center for $ ( \IE, \cB ) $} (or {\em $ \cB $-permissible for $ \IE $}) if $ D $ is permissible for $ \IE $ and additionally $ D $ has at most simple normal crossings with $ B $.
	We say that a blowing up is $ \cB $-permissible if its center is.
		
	\medskip
	
	\item
	The transform of $ ( \IE, \cB ) $ under a $ \cB $-permissible blowing up is defined as $ ( \IE', \cB' ) $, where $ \IE $ is the transform as defined in Definition \ref{Def:perm} and $ \cB' = \widetilde{\cB} \cup \{ E' \} $, where $ \widetilde{\cB} $ is the collection of the strict transforms of the elements of $ \cB $ and $ E' $ denotes the exceptional divisor of the blowing up. 
		
	\medskip
	
	\item
	A {\em (local) resolution of singularities for $ ( \IE, \cB ) $} is a finite sequence of local $ \cB $-permissible blow-ups such that the final transform of $ \IE $ is resolved.
	\end{enumerate}
\end{Def}

\medskip

Since the center in (2) is $ \cB $-permissible, the divisor $ B' $ is automatically a simple normal crossing divisor.
Further, note that a resolution of singularities for $ ( \IE, \cB ) $ yields one for $ \IE $ whereas the converse is not necessarily true in general (even if $ \cB = \varnothing $) since we do not require the centers to be $ \cB $-permissible in the latter.

This kind of additional structure $ \cB $ given by some simple normal crossing divisor appears for example as the collection of the exceptional divisors of preceding blowing ups (or maybe one already wants to start the resolution process with this extra structure).

\begin{Ex}
\label{Ex:easy_B}
	Let $ R = k[[x,y,z]] $, where $ k $ is any field, and let $ J = \langle x^2 - y^3 z^2 \rangle $.
	We have $ \Sing ( J , 2 ) = V ( x, y ) \cup V( x, z ) $.
	In order to get a canonical resolution of singularities for $ (J,2) $ (resp.~$ J $) we have to blow up the origin $ V ( x, y, z) $.
	If we consider the transform of $ ( J, 2 )$ at the origin of the $ Z $-chart
	(i.e., at the point with coordinates $ (x',y,',z') := (\frac{x}{z}, \frac{y}{z},z)$), then we get the same pair. 
	Hence there is no obvious improvement.
	
	If we add the extra structure of the exceptional divisor, $ E' = V ( z' ) $, then it is clear that $ V (x', z' ) $ is contained in the exceptional divisor, whereas $ V( x',y') $ is not.
	Thus $ V( x',y') $ is the strict transform of the corresponding component before the blowing up and it is in some sense ``older" than $ V ( x', z') $.	
	
	If we have some extra structure at the beginning given by the regular divisor $ B = V( y + z^2 ) $ then we see that $ V ( x,y ) $ is {\em not} $ \cB $-permissible for $( J , 3 ) $.
\end{Ex}

If we desire to obtain a canonical resolution in the last example, we need to be able to distinguish the two components in the singular locus;
otherwise we would always blow up the origin and the process never ends. 
Therefore it is necessary to take the exceptional divisors of the resolution process into account. 

Another reason is that for applications it may become necessary that the preimage of a given singular variety/scheme under a resolution of singularities is a simple normal crossing divisor.
Since the exceptional divisors appear in the preimage, we have to ensure that they behave nicely during the resolution process.

Let us mention the following example which explains also the handling in the general case: 
Consider a curve embedded in a three dimensional regular ambient space and suppose we have resolved its singularities and the exceptional divisors have at most normal crossings with it. 
The preimage of the original variety is not yet a divisor. 
In order to achieve this one has to blow up the entire ($\cB $-permissible) curve itself at the end.
But note that condition (2) of Conjecture \ref{Conj:ERS} is then not fulfilled anymore.

%
%
%
%
%
%
%
%		Resolution invariant
%
%
%
%
%
%
%

\medskip

\section{A Characteristic Zero invariant}

In order to make more precise what we mean by improving the singularity, we briefly give the definition of the invariant by Bierstone and Milman \cite{BMfunct} for resolution of singularities over fields of characteristic zero in a variant adapted to our setting (with no restriction to the characteristic), cf.~\cite{BerndBM}.

Let $ X \subset Z $ be a closed subscheme of an excellent regular Noetherian scheme $ Z $ and let $ \cB $ be a finite set of irreducible divisors on $ Z $ such that the reduced divisor $ B $ defined their union has at most simple normal crossing singularities.
The invariant consists of a finite set of pairs $ ( \nu_i , s_i ) $, where $ \nu_i $ is the order of some pair and $ s_i $ counts certain divisors in $ \cB $,
$$
 \iota_X (x) := ( \nu_1 (x), s_1(x) ; \, \nu_2 (x), s_2(x): \ldots ) ,
 \ \ 
 \mbox{ for }
 \ x \in X,
$$ 
where we equip the string of numbers with the lexicographical order. 
Further, we denote by $  \iota_X (x)_{\kappa + \half} :=  ( \nu_1 (x), s_1(x) ; \ldots; \nu_{\kappa + 1} (x) ) $ the truncation of $ \iota_X(x) $  after the $ (2\kappa + 1) $-th entry.
The first invariant, is the Hilbert-Samuel function of $ X $ at $ x $ (\cite{BMfunct} section 1.3, see also the modified version \cite{CJS} Definition 1.28)
$$ 
	\nu_1 (x) := H_X ( x ) .
$$ 

It is possible that there are preceding blowing ups in the resolution process.
Let us fix some notation,

\begin{equation}
\label{eq:bigseq}
\begin{array}{ccccccccccccc}
\cB_0	&	& \cB_1	&	&	\ldots &	&	\cB_{i} &	&	\ldots &	&	\cB_{j-1}	& &	\cB_j =: \cB \\[8pt]
Z_0					& \stackrel{\pi_1}{\longleftarrow} &	Z_1	& \stackrel{\pi_2}{\longleftarrow} &	\ldots & \stackrel{\pi_{i}}{\longleftarrow} &	Z_{i} & \stackrel{\pi_{i + 1 }}{\longleftarrow}  & \ldots 	& \stackrel{\pi_{j-1}}{\longleftarrow} &	Z_{j-1}		& \stackrel{\pi_j}{\longleftarrow} 	&	Z_j =: Z	\\[5pt]
\bigcup	&	& \bigcup	&	& &	&	\bigcup &	& &	&	\bigcup	& &	\hspace{1cm}\bigcup\\[5pt]
X_0					& \longleftarrow &	X_1	& \longleftarrow &	\ldots	& \longleftarrow &	X_{i}	& \longleftarrow &	\ldots	& \longleftarrow &	X_{j-1}		& \longleftarrow 	&	X_j	=:X\\[5pt]%
& &	&  & &  & x_i & \leftmapsto & \ldots	& \leftmapsto &	 x_{j-1} & \leftmapsto &  x_j =: x,
\end{array}
\end{equation}
where $ X_0 \subset Z_0 $ and $ \cB_0 $ is the situation that the resolution process initially started with 
($ X_0 $ a closed subscheme of an excellent regular Noetherian scheme $ Z_0 $ and $ \cB_0 $ a finite set of irreducible divisors having simple normal crossings),
$ X_i $ (resp.~$ Z_i $) is the strict transform of $ X_{i-1} $ (resp.~$ Z_{i-1} $) under the blowing up $ \pi_i $, $ \cB_i $ is the union of the strict transforms of the components in $ \cB_{i-1} $ and the exceptional divisor of $ \pi_i $.

Further, $ x \in X $ is the point that we consider.
Thus $ R = \cO_{Z,x} $ and we associate $ \IE := \bigcap\limits_{i=1}^m ( f_i, b_i ) $ to $ X $ via choosing a standard basis for the ideal locally defining $ X $ at $ x $, as in Remark \ref{Rk:IE_to_x}.
We denote by $ \cB(x) $ the set of divisors in $ \cB $ passing through $ x $.
 
\begin{Rk}[\em Old and new exceptional divisors]
	Let $ H $ be an invariant measuring the complexity of the singularity of $ X $ at $ x $.
	We require that 
	$ H $ is upper semi-continuous
	(so that the locus where it is maximal is closed), 
	and 
	that $ H $ does not increase if we blow up a regular center contained in the locus where $ H $ is maximal.
	Assume that the centers of blowing ups in \eqref{eq:bigseq} are of the preceding type.
	We distinguish the components in $ \cB(x) $ as follows:
	
	We choose $ i \in \{ 0 , \ldots, j \} $ such that $ H $ last decreased after the $ i $-th blowing up, $ H(x_{i-1}) > H(x_i) = H(x) $, where $ x_\kappa = \pi_{\kappa+1} (x_{\kappa + 1}) $ for $ i \leq \kappa < j $. 
	We declare all divisors in $ \cB_i (x_i)$ to be \textit{old} \wrt $ H $, $ O_H(x_i) := \cB_i(x_i) $.
	This leads to the distinction $ \cB(x) = O_H(x) \cup N_H(x) $,
	where $ O_H(x) $ are the strict transforms of $ O_H(x_i) $, and $ N_H(x) := \cB(x) \setminus O_H(x) $ are said to be \textit{new} since they arose after $ H $ decreased the last time.	
\end{Rk}

\begin{Constr}
	\label{Constr.}
	Let $ x \in X \subset Z $ as before.
	Using the previous notations, we set
	$$
	 \IG_1(x) := \IE = (f_1, b_1) \cap \ldots \cap (f_m,b_m)
	$$
	and $ R $ is the local ring of $ Z$ at $ x $.
	Note that $ \ord_x ( \IE) = 1 $.
	We define
	$$
		s_1 (x) := \# O_{\nu_1} (x) .
	$$
	Let $ g_1, \ldots, g_{s_1} \in R $ be local generators for the old divisors.
	We put
	$$
		\IF_1 (x) := \IG_1(x) \cap ( \langle g_1, \ldots, g_{s_1} \rangle , 1 ),
	$$ 
	and we replace $ \cB(x) $ by $ \cB^{(1)}(x) := N_{\nu_1} ( x) $.
	If there exists no hypersurface of maximal contact for $ \IF_1 (x) $, we stop and set $ \iota_X(x) := (\nu_1(x), s_1(x) ) $.
	
	Suppose $ V (z_1) $ has maximal contact with $ \IF_1(x) $.
	Let $ ( w) = (w_2, \ldots, w_n) $ be any set of regular elements extending $ z_1 $ to a \RSP for $ R $.
	We pass to the coefficient pair
	$$
		\ID_1 (x) : = \ID (\IF_1(x);  w; z_1 ) =: (I_1,d_1).
	$$
	on $ R_1 = R/ \langle z_1 \rangle  $. 
	We may consider $ \cB^{(1)}(x) $ as a set of divisors in $ \Spec(R_1) $.
	Denote by $ \eta_1, \ldots, \eta_a $ the generic points of the components in $ \cB^{(1)}(x) $.
	We define 
	$$
		\nu_2(x) := \ord_x (\ID_1(x)) - \sum_{i=1}^a \ord_{\eta_i} (\ID_1(x)) \geq 0 ,
	$$
	%and
	$$
		\iota_X(x)_{1 + \half} := (\nu_1(x), s_1(x);\, \nu_2(x) ).
	$$
	We have a factorization $ I_1 = M(I_1)\cdot N(I_1) $,
	where $ M(I_1) $ is a monomial in the components coming from $ \cB^{(1)} (x) $ 
	(that corresponds to the sum that we subtract in the definition of $ \nu_2(x) $)
	and $ N(I_1) $ is an ideal in $ R_1 $ such that none of the components in $ \cB^{(1)}(x) $ factors from it.
	We associate the so called {\em companion pair} to this:
	$$
		\IG_2 (x) 
		:= 
		\left\{ 
		\begin{array}{ll}
			(N(I_1), d \cdot \nu_2(x)),
			&
			\mbox{if } \nu_2(x) \geq 1 ,
			\\[5pt]
			(N(I_1), d \cdot \nu_2(x))
			\cap 
			(M(I_1), d \cdot (1 - \nu_2(x) ) ),
			&
			\mbox{if } \nu_2(x) < 1 .
		\end{array}
		\right.
	$$
	We repeat the previous steps for $ \IG_2(x) $ instead of $ \IG_1(x) $. 
	The old components in $ \cB^{(1)}(x) $ are now considered to be old \wrt $ \iota_X(x)_{1+\half } $.
	This ends when either there does not exist a hypersurface of maximal contact,
	or $  M(I_\kappa ) = I_\kappa $, for some $ \kappa \geq 1 $.
	In the latter case, $ \nu_{\kappa +1 }(x) = 0 $ and this is called the {\em monomial case}, where a simple combinatorial procedure can be applied to lower the invariant
	(see \cite{BMfunct} section 5 Step II Case A).
\end{Constr}

\begin{Rk}
	While maximal contact always exists locally over a field of characteristic zero, this is not true in general.
	Hence $ \iota_X(x) $ can only be the beginning of an appropriate invariant, in general, and there is a need to find a second part controlling the case where there is no maximal contact.
	
	\smallskip
	
	The constructed invariant has the following properties
	(see \cite{BMfunct} Theorem 7.1):
	(1) $ \iota_X $ distinguishes regular and singular points, % Hilbert-Samuel function does so
	(2) $ \iota_X $ is upper semi-continuous,
	% s_1 is a history function as in CJS
	% the order is usc ($ \cB = \emptyset $ case), and the subtracted part is an invariant which does not change under passing to a close by point "delta - 1 " unless we loose the component, but then $ s_1 $ drops too.
	(3) $ \iota_X $ can not increase under a blowing up with a center contained in the maximal locus of $ \iota_X $, %maximal contact and \nu_2 is the order of the strict transform over $ d_1 $.
	% Note if $ \nu_2 < 1 $ and we are not in a chart where order of $ N(I_1) $ drops, we must be in a chart of $ M(I_1) $ and hence less old exceptionals.
	(4) $ \iota_X $ can only decrease finitely many times strictly until the singularities are resolved. % values in a discrete subset of $ \IQ^{2n} $.
\end{Rk} 

%
%
%
%
%
%
%%%%		First examples
%
%
%
%
%
%
%

\medskip

\section{First examples for the reduction}

Let us discuss some examples before coming to the precise formulation of the reduction result.
All examples are of the same shape:

The ambient scheme is $ Z = \Spec (S[y]) $,
where $ (y) = (y_1, \ldots, y_r) $ and 
$ S $ is a regular local ring with parameters $ ( u) = (u_1, \ldots, u_e) $
(e.g., $ S = k[u]_{\langle u \rangle} $, for some field $ k $, or $ S = \IZ[u_2,\ldots, u_e]_{\langle p,u_2, \ldots, u_e \rangle} $, for $ p \in \IZ_+ $ prime);
$ \cB= \varnothing $ and
$ X = V ( f ) $ is a hypersurface in $ Z $,
for
$$
	f = y^N + h(u) 
	\in \langle u ,y \rangle^{|N|},
	\ \ \
	\mbox{ where }
	\
	h \in \langle u  \rangle^{|N|+1}
	\
	\mbox{ and }
	\
	|N| \geq 2.
$$
Further, $ R = S[y]_{\langle u, y \rangle } $
is the local ring of $ Z $ at the closed point $ x_0 = V( u,y ) $ (with maximal ideal $ M = \langle u ,y \rangle $) and $ p > 0 $ denotes the characteristic of the residue field of $ R $.

\begin{Ex}
	\label{Ex:More_detail_Example}
 Let $ r= 1 $, $ p > 2 $, and let $ X = V( f ) \subset Z $ be defined by
 $$
	f := y^2 + h ( u_1, u_2, u_3 ) .
 $$
 Therefore we consider the pair $ \IE := ( f , 2 ) $ in $ R $. %which is the local ring of $ Z $ at the closed point $ x_0 $.
 If we apply the derivative by $ y $, we obtain the equivalences (using Facts  \ref{Facts:Id_Exponents})
 \begin{equation}
 \label{eq:equivlanece_in_example}
 	( f, 2 ) \sim  ( 2y, 1 ) \cap ( f, 2 )  \sim ( y, 1 ) \cap ( f, 2 ) \sim ( y, 1 ) \cap ( h, 2 ),
 \end{equation}
 where we use in the second equivalence that $ 2 $ is invertible in $ R $.
 Hence $ V ( y ) $ has maximal contact with $ \IE $ (Definition \ref{Def:max_contact})
 and $ \iota_X(x_0)_{1 + \half}  = (2,0; \frac{m_0}{2}) $, where $ m_0 := \ord_{M}(h) $. 
 This allows us to reduce the problem of finding a resolution for  $ \IE $ (resp.~for $ X $) to the corresponding problem for the surface $ Y = V ( h )  \subset \Spec (S) $  determined by the coefficient pair $ \ID ( \IE; u ;y) = (h, 2 ) $.
 
 Since resolution of singularities via blowing ups in regular centers is known for surfaces \cite{CJS}, we can proceed as follows:
 By induction on the dimension, we can resolve $ \ID^* := ( h, m_0 ) $, i.e., if we denote by $ (h_1, m_0 ) $ the transform of $ \ID^* $, say in $ S_1 $ locally at a point $ x_1 $ corresponding to the ideal $ M_1 $, then $ \Sing (h_1, m_0 ) = \varnothing $ or, equivalently,  $ \ord_{M_1} (h_1) < m_0 $.
 Note that we can lift the sequence of blowing ups as explained in Remark \ref{Rk:Blowup_on_V(y)_back_to} to one in $ R $ and we denote by $ X_1 $ the strict transform of $ X $.
 
 We distinguish the parameters in $ S_1 $ as $ ( u, v ) = (u_1, \ldots, u_e; v_1, \ldots, v_d ) $ such that each $ u_i $ corresponds to an exceptional divisor passing through $ x_1 $, while $ (v) $ is any set of elements extending $ ( u ) $ to a system of parameters for $ R_1 $.
 (By abuse of notation, we continue to use the letters $ u_i $ instead of $ u_i' $).
  
 Note that the transform of $ ( h,2) $ under the preceding blowing ups is $ ( \epsilon u^A h_1, 2 ) $, for some $ A = (A_1, \ldots, A_e) \in \IZ^{e}_{\geq 0} $ and a unit $ \epsilon \in R_1^\times $
 (which comes from exceptional divisors not containing $ x_1 $).
 Thus $ \iota_{X_1} (x_1)_{1+ \half} = (2, 0; \frac{m_1}{2}) <  \iota_{X} (x_0)_{1+ \half} $,
  where $ m_1 := \ord_{M_1} (h_1) $.
 If $ \Sing(h_1,2) \neq \varnothing $, we repeat the previous step and resolve $ ( h_1, m_1) $.
 Since $ m_1 < m_0 $, we reach after finitely many steps $ \Sing(h_\ell,2) = \varnothing $, for $ \ell \geq 1 $.
 Without loss of generality, we assume that $ \Sing(h_1,2) = \varnothing $.

 If $ \Sing ( \epsilon u^A h_1, 2 )  = \varnothing $, then 
 \eqref{eq:equivlanece_in_example} provides that we obtain a resolution of $ (f,2) $.
 Suppose this is not the case.
 Then we have to pass to the companion pair,
 $$
	 \IG := (h_1, d) \cap (u^A, 2 - d),
	  \ \
	  \mbox{ with } 
	  \ d := \ord_{M_1} (h_1) \leq  1.
 $$
 (If $ d = 0 $, we neglect $ ( h_1, d ) $).
 In other words, we also take the monomial $ u^A $ into account in order to obtain centers that are permissible for the transform of $ ( f, 2) $.
 We then resolve $ \IG $ which will either result in a decrease of the order of $ h_1 $ (so that we are left with the exceptional monomial) or in a decrease of $ A $.
 By lifting the sequence of blowing ups, we eventually obtain a resolution for $ ( f, 2) $.
 Alternatively, it is also not too hard to construct an explicit resolution for $ (f_1, 2) = (y^2 +  u^A h_1, 2 ) $ with $ \ord_{M_1} (h_1) \leq 1 $ (for example, using ideas of \cite{BMbinomial} or \cite{KollarToroidal}).
\end{Ex}

\begin{Ex}
	Suppose $ S = \IZ_{(p)} $, for $ p $ prime.
	Let $ r = p $ and $ X = V(f) \subset Z = \Spec (S[y_1,\ldots, y_p]) $ be defined by
	$$
		f := p^p + y_1 y_2 \cdots y_p.
	$$
	Using Facts \ref{Facts:Id_Exponents}, we obtain:
	$$
		\IE:= (f,p) \sim (p^p  + y_1 y_2 \cdots y_p, p ) \cap (\langle 
		y_1, \ldots, y_p \rangle , 1 ) 
		\sim (\langle p, y \rangle , 1 ). 
	$$
	Therefore $ V(p, y ) = \Sing (\IE) $ is the only permissible center. 
	After blowing it up, the idealistic exponent is resolved in every chart,
	while the strict transform $ X ' $ of $ X $ is still singular.
	For example, at the point with coordinates $ (z',y_1',\ldots, y_p') = (\frac{p}{y_1},y_1, \frac{y_2}{y_1},\ldots, \frac{y_p}{y_1}) $, we get $ X ' = V(f') $ with $ f' = z'^p + y_2' \cdots y_p' $.
	(Note that the ambient scheme is now $ Z' =  \Spec ( \IZ_{(p)}[z', y']/\langle p - z' y_1 \rangle) $).
	Nonetheless, since the order is at most $ p - 1 $, one can achieve a resolution of $ X $ with the methods of characteristic zero.
\end{Ex}

\begin{Ex}
 Let $ r = 1 $, $ p > 2 $, and let $ X = V( f ) \subset Z = \Spec (S[y]) $ be defined by
 $$
	f := y^4 + h ( u_1, u_2, u_3 ),
	 \ \
	 \mbox{ with }
	 \
	 h \in \langle u \rangle^5. 
 $$
 Since $ 4 $ is a unit in $ S $ we get as before
 $$
 	\IE := (f, 4 ) \sim (y, 1) \cap ( f, 4)  \sim (y, 1) \cap ( h, 4).
 $$
 In particular, $ V ( y ) $ has maximal contact with $ \IE $.
 This reduces the resolution problem of $ X $ to that of $ Y = V ( h ) \subset \Spec(S) $
 and we can resolve $ ( f, 4 ) $ with the methods of the previous example.
 But this time we do not obtain a full local resolution of singularities for $ X $
 since the singularity given by the strict transform $ f' $ of $ f $ is not necessarily regular.
 We only know that the order at every point is below $ 4 $.

 As a special case of this example, the strict transform of $ f $ could be of the form $ f' = u_1^3 + \widetilde{h}(u_1, u_2, u_3, y) $ with $ \widetilde{h} \in \langle u_1, u_2, u_3, y \rangle^4 $. 
 If $ p = 3 $, we can not deduce that $ V ( u_1 ) $ has maximal contact with the techniques of the previous section (which is in general false, even if there is no $ u_1 $ appearing in $ \widetilde{h} $).
\end{Ex}

\medskip

Nevertheless, we see that there exist special cases in which we can reduce the resolution problem for a given singularity to one in lower dimensions.
The last example can be formulated more generally as 

\begin{Ex}
	Let $ r = 1 $, $ n \in \IZ_{\geq 2} $ with $ (n,p)= 1 $ and let $ X = V( f ) \subset Z = \Spec (S[y]) $ be defined by
 $$
	f := y^n + h_2(u) y^{n-2} + \ldots +  h_n ( u),
	 \ \
	 \mbox{ with }
	 \
	 h_i(u) = h_i( u_1, u_2, u_3 ) \in \langle u \rangle^{i+1}. 
 $$
Since $ (n,p)= 1 $ , $ n \in S^\times $ is a unit.
 Therefore $ V ( y ) $ has maximal contact at the at  $ x_0 = V(u,y) $,
  $$
  	\IE := (f, n ) \sim ( y , 1) \cap ( f, n)  \sim (y, 1) \cap  \bigcap_{i=2}^n ( h_i, i ) = (y, 1) \cap \ID(\IE;u;y)
  $$
  Since $ S $ has dimension three, we can resolve $ \ID(\IE;u;y) $ as before.
  Using the techniques of Example \ref{Ex:More_detail_Example} in this slightly more general variant, we obtain a resolution for $ \IE $ and hence an improvement of the singularities of $ X $, i.e., a decrease of $ \iota_X(x)_{1+\half} $.
\end{Ex}

\medskip

Sometimes this reduction can also be applied for higher dimensional varieties.

\begin{Ex}
	Let $ r = 2 $ and let $ X = V( f ) \subset Z = \Spec (S[y_1,y_2]) $ be defined by
 $$
	f := y_1^p \, y_2^p + h ( u_1, u_2, u_3 )
	\ \
	\mbox{ with }
	\
	h \in \langle u \rangle^{2p+1}.
 $$
 Using the differential operator $ \dell{y_1^p } $ (resp.~$ \dell{y_2^p } $) which sends $ y_1^p $ (resp.~$ y_2^p $) to one (see \eqref{eq:HasseSchmdit} below for more on these) one can show that $ V ( y_1, y_2 ) $ has maximal contact with $ \IE := ( f, 2p ) $, and
 $$
 	 (f, 2p ) \sim ( y_1, 1) \cap ( y_2, 1) \cap ( h, 2p)  .
 $$ 
 Again, we reduced the problem of lowering the order of $ f $ to the problem of finding a resolution for $ ( h, 2p) $ in $ S $.
 As before, we can achieve that the invariant drops below $ \iota_X (x_0)_{2+\half} = (2p,0;1,0;\frac{m_0}{2p +1 }) $, where $ m_0 := \ord_{M}(h) $.
 (Note that the case $ p = 2 $, i.e., order $ 4 = p^2 $, is not excluded).
\end{Ex}

\begin{Ex}
	\label{Ex:moresteps1}
	Let $ r = 3 $, set $ z_3 := y_3 $, and let $ X = V( f ) \subset Z = \Spec (S[y_1,y_2,z_3]) $ be defined by
 $$
	f := y_1^p y_2^p +  ( z_3^{ 2p +  1}  + h ( u_1, u_2, u_3 ) ) = 0
	\ \
	\mbox{ with }
	\
	h \in \langle u \rangle^{2p+2}.
 $$
 	As explained in the previous example the resolution of $ \IE := (f, 2p ) $ reduces to that of $ \ID_1 := ( z_3^{ 2p +  1}  + h ( u_1, u_2, u_3 ) \,, 2p ) $.
 	Since $ 2p + 1 $ is prime to $ p $, we can reduce the resolution problem for $ \ID_1^\ast:= ( z_3^{ 2p +  1}  + h ( u_1, u_2, u_3 ) \,, 2p + 1 )  $ further to the one for $ ( h ( u_1, u_2, u_3 ) \,, 2p + 1 ) $.
 	The latter is in $ S $ and hence can be solved as in Example \ref{Ex:More_detail_Example}, and the same techniques show that the original singularity improves, i.e., the invariant drops below $ \iota_X(x_0)_{2+\half} = (2p, 0; 1,0; \frac{2p+1}{2p}) $ after finitely many permissible blowing ups. 
\end{Ex}

\begin{Ex}
\label{Ex:Not_full_Reso}
	Let $ r = 2 $, set $ z := y_2 $, and let $ X = V( f ) \subset Z = \Spec (S[y, z]) $ be defined by
	$$
		f := y^{p^2 + 1 } + z^{2p^2 + 1 } + h (u_1, u_2) 
		\ \
		\mbox{ with }
		\
		h \in \langle u \rangle^{2p^2+2}.
	$$
	We have $ \IE: = ( f, p^2 + 1 ) \sim ( y, 1 ) \cap ( z^{2p^2 + 1 } + h, p^2 + 1 ) $ and further	$ \ID_1^* :=  ( z^{ 2 p^2 + 1 } + h, 2 p^2 + 1 ) \sim ( z, 1 ) \cap ( h, 2 p^2 + 1 ) $.
	Thus $ \iota_X(x_0)_{1+\half} = (p^2 + 1, 0 ; \frac{2p^2 + 1}{p^2 + 1}) $.
	Any permissible center will be of the form $ V ( y, z, g_1 , \ldots, g_a  ) $ for certain $ g_i \in S $.
	If we blow up and consider the origin of the $ Z $-chart, then
	$ \IE' =  ( y^{p^2 + 1 } + z^{p^2 } + \widetilde h (u_1, u_2, z ) , p^2 + 1 ) $ is resolved.
	But in the next step of the resolution process for $ X $, we have to deal with
	$$
		\IE^{new} := ( z^{p^2} + y^{p^2 + 1 } + \widetilde h (u_1, u_2, z ) ,\, p^2 ) 
	$$ 
	and it is not clear why $ \IE \sim ( z, 1 ) \cap \IE$ should be true.
	(In general, this is false!)
\end{Ex}

 Narasimhan gave in \cite{Nara} an example of a threefold ($ t^2 + x y^3 + y z^3 + x^7 z = 0 $ over a field of characteristic $ p = 2 $) whose singular locus has embedding dimension four.
Hence there can not exist a regular hypersurface containing the singular locus.
In particular, there is no regular hypersurface of maximal contact.

Let us mention the following construction in connection with this.

\begin{Ex}
	Let $ m \in \IZ_+ $ and $ f_0, f_1, \ldots, f_m \in S[z_1, \ldots, z_m] $
	with $ f_i \in \langle z, u_1, \ldots, u_{e} \rangle $.
	The singular locus of the hypersurface $ X = V( f ) \subset Z = \Spec (S[z] ) $ defined by 
	$$	
		f := f_0^p + z_1 f_1^p + \ldots z_m f_m^p \, ,
	$$ 
	is $ \Sing ( X ) = V ( f_0, f_1, \ldots, f_m ) $.
	Using this, we can easily produce other examples, where the embedding dimension of the singular locus (resp.~of the locus of order $ p $) is $e + m $.	
\end{Ex}

%
%
%
%
%
%
%		Reduction 
%
%
%
%
%
%
%

\medskip

\section{The criterion for a possible reduction}

In this section we generalize the second part of Proposition \ref{Prop:ItcDirRidUnique} to a characteristic-free variant. 
By translating the equivalence to $ R $ we deduce the criterion to decide when a local resolution problem can be reduced to one in lower dimensions.

Recall that two pairs $ \IE_1 $ and $ \IE_2 $ on some ring $ S $ are equivalent, $ \IE_1 \sim \IE_2 $, if the following holds (see Definition \ref{Def:sim}):
For any finite system of independent indeterminates 
$ (t) = (t_1, \ldots, t_a ) $, 
any sequence of local blowing ups (over $S[t]$) which is permissible for $  \IE_1[t] $ is also permissible for $ \IE_2[t] $ and vice versa
(where $ \IE_i [t] = (J_i \cdot S[t], b) $ for $ \IE_i= (J_i,b_i) $ and $ J_i \subset S $).

\begin{Thm}
\label{Thm:Equiv}
	Let $ R $ be a regular local ring
	and $ \IE_\sim $ an idealistic exponent on $ R $.
	Assume $ \Sing ( \IE_\sim ) \neq \varnothing $.
	Then we have that
	\begin{equation}
	\label{eq:equiv_IRid_IT}
		\IRid( \IE_\sim ) \sim \ITC(\IE_\sim ).
	\end{equation}  
	If furthermore the reduced ridge coincide with the directrix, then we even obtain the equivalence $ \IDir (\IE_\sim ) \sim \IRid ( \IE_\sim )  \sim \ITC(\IE_\sim ) $.
\end{Thm}

\medskip

\smallskip 

Recall that $ \IRid(\IE_\sim) $ and $ \IT (\IE_\sim) $ are idealistic exponents in the graded ring $ gr_M (R) $. 
If $ (w) = (w_1, \ldots, w_n ) $ is a \RSP for $ R $, then we have
$ gr_M(R) \cong k[W_1, \ldots, W_n ] =: S $, where $ k = R/M $ is the residue field of $ R $.

\smallskip 

The main ingredient for the proof of the above proposition are $ k $-linear differential operators $ \Diff_k(S) $ in $ S $ and the Diff~Theorem, Facts \ref{Facts:Id_Exponents}(3).
If $ \car{k} = 0 $, then $ \Diff_k(S) $ is generated by the derivations. 
This is not true if $ \car{k} = p > 0 $, where also 
{\em Hasse-Schmidt derivations} appear
(\cite{GiraudMaxPos}, sections 2.5, 2.6, 
or \cite{BernhardThesis} Example (3.3.1)):
Let $ f (W) \in S $ and $ (T) = (T_1, \ldots, T_n) $ be a set of independent variables. 
We have
$$
	f(W+T) = f(W_1+T_1, \ldots, W_n + T_n)
	= \sum_{N \in \IZ^n_{\geq 0}} f_N(W) \, T^N
	\in k[W,T].
$$
Given $ N = (N_1, \ldots, N_n ) \in \IZ^n_{\geq 0} $, the Hasse-Schmidt derivative of $ f $ with respect
$ W^N = W_1^{N_1} \cdots W^{N_n} $ is defined as
$$
	\cD_N (f) := 
	\frac{\partial}{\partial W^N} f := f_N(W ) \in S .
$$
Hence we have that, for any $ B , N \in \IZ^n_{\geq 0} $ and $ \lambda \in k $,
\begin{equation}
	\label{eq:HasseSchmdit}
	\cD_N ( \lambda W^B ) = \lambda \binom{B}{N}  W^{B-N}, 
\end{equation}
where $ \binom{B}{N} = \binom{B_1}{N_1} \cdots \binom{B_n}{N_n} $, and $ \binom{B_i}{N_i} = 0 $ if $ N_i > B_i $.
In particular, $ \frac{\partial}{\partial W_1^q} W_1^q = 1 $, for any $ p $-power $ q $.
(Neglecting the characteristic, we can write the Hasse-Schmidt derivative as
$ \frac{\partial}{\partial W^N} = \frac{1}{N_1! \cdots N_n!}\cdot \big( \frac{\partial}{\partial W_1} \big)^{N_1} \cdots \big( \frac{\partial}{\partial W_n} \big)^{N_n} $).

\begin{proof}[Proof of Theorem \ref{Thm:Equiv}]
	By Proposition \ref{Prop:ItcDirRidUnique}, we may assume $ \car{k} = p > 0 $, where $ k = R/M $.
	Let $ \IE = ( J, b ) $ be a representative of $ \IE_\sim $.
	We set $ I := In_M (J,b) $.
	Let $ ( Y ) = ( Y_1, \ldots, Y_r ) $ (resp. $ ( \sigma ) = ( \sigma_1, \ldots, \sigma_s ) $) be elements defining the directrix (resp.~the ridge) of $ \IE $.
	We pick $ (\sigma) $ in triangular form \eqref{eq:triangular}.
	We have $ \IT (\IE) = ( I, b ) $, and $ \IRid (\IE) = \bigcap_{i=1}^s ( \sigma_i, q_i ) $, and $ gr_M (R) \cong k[ Y, U ] $, where $ ( U ) = ( U_1, \ldots, U_e ) $ is a system of elements extending $ ( Y ) $ as before. 
	Set 
	\begin{center}
		$ \tau := \sigma_1 $, $ q := q_1 $, and $ Z := Y_1 $.
	\end{center}
	
	{\em Case $ s = 1  $:} 
	By the definition of the ridge, there exists a system of generators for $ I $ such that every element $ F $ in it is of the form $ F = \lambda_{m} \tau^m $, for some $ \lambda_{m} \in k^\times $ and $ m = \frac{b}{q} \in \IZ_+ $.
	Facts \ref{Facts:Id_Exponents} implies $ ( F, b ) \sim ( \tau , q ) $ and thus
	$
		\IT(\IE) = (I,b) \sim (I,b) \cap (F,b) \sim (I,b) \cap ( \tau, q ) \sim (\tau, q ) = \IRid (\IE).
	$

	\smallskip 	
	
	{\em Case $ s \geq 2 $}:
	By the definition of the directrix, we can choose as system of generators for $ I $ such that any element $ F $ in the system  is
	homogeneous of degree $ b $ \wrt $ ( Y ) $. 
	Further, $ F $ can be written as $ F = \sum C_{A,d} \tau^d \sg_+^A $, $ d \in \IZ_{\geq 0} $, $ A \in \IZ_{\geq 0}^{s-1} $, $ ( \sigma_+ ) := (\sigma_2, \ldots, \sigma_s ) $.
	There exists at least one generator for which there is some $ d_0 \in \IZ_+ $, $ ( d_0, p ) = 1 $, and  $ C_{A^0, d_0 } \neq 0 $, for some $ A^0 \in \IZ^{s-1}_{\geq 0} $;
	otherwise $ ( \sigma ) $ would not determine the ridge and have to replace $ \tau $ by $ \tau^p $.
	
	We pick such a generator $ F $ and choose $ d_0 \in \IZ_+ $, $ ( d_0, p ) = 1 $, to be maximal for $  F $.
	We put $ \beta := q \cdot (d_0 - 1) $.
	Using Lucas' Theorem (\cite{Lucas} section XXI,
	$ \binom{\sum_{i=1}^\kappa a_i p^i}{\sum_{i=1}^\kappa b_i p^i} \equiv \prod_{i=1}^\kappa  \binom{a_i}{b_i} \mod p $, for $ a_i, b_i \in \{ 0, \ldots, p-1\} $),
	we have
	$$
		\cD_\beta F = 	d_0 \cdot \tau \cdot ( \sum_{A^0} C_{A^0,d_0} \, \sigma_+^{A^0} ) + \sum_{A} C_{A,d_0 - 1}\, \sigma_+^{A} \,,
	$$ 
	and there is at least one coefficient $ C_{A^0,d_0} \neq 0 $.
	We pick one of these $ A^0 $ and apply the Hasse-Schmidt derivation $ \cD_0 := \frac{\partial}{\partial (Y_2 \cdots Y_s)^{q\ast A^0}} $, where $ (Y_2 \cdots Y_s)^{q\ast A^0} = Y_2^{q_2 A_2^0} \cdots Y_s^{q_s A_s^0} $ for $ A^0 = ( A_2^0, \ldots, A_s^0 ) \in \IZ^{s-1}_{ \geq 0} $.
	Then we obtain that
	$$
		\cD_0 \cD_\beta F = 	d_0 \cdot \tau + \sum_{A} C_{A,d_0 - 1} \cdot \binom{A}{A_0} \cdot \sigma_+^{A - A^0} \,,
	$$ 
	which is homogeneous of order $ q $.
	Since $ d_0 $ is invertible in $ k $, we can define $ s_1 := {d_0}^{-1}\cdot \sum_{A} C_{A,d_0 - 1} \cdot \binom{A}{A_0} \cdot \sigma_+^{A - A^0} $ and get 
	$$
		 {d_0}^{-1}\cdot  \cD_0 \cD_\beta F =  \tau + s_1 \,.
	$$
	Therefore, we have $ (I,b ) \sim ( I, b ) \cap  ( \tau + s_1 , q ) $.
	Using Hasse-Schmidt derivations and Facts \ref{Facts:Id_Exponents}, we can obtain that $ s_+ $ is an additive polynomial.
	Without loss of generality we may assume $ s_+ = 0 $.
	(If this is not the case, we put $ \tau^{new} := \tau + s_+ $ and replace $ \tau $ by $ \tau^{new} $; at the end this will not change the result and  this still defines the ridge).
	Therefore we have 
	\begin{equation}
	\label{eq:first_step_proof}
		\IT(\IE) = (I,b) \sim (I,b) \cap (\tau, q) = (I,b) \cap (\sigma_1, q_1).
	\end{equation}
	
	In the next step, we repeat the previous for $ \tau = \sigma_2 $.
	This provides 
	$$
		\IT(\IE) \sim  (I,b) \cap (\sigma_1, q_1) \cap (\sigma_2, q_2 ).
	$$
	Note that we possibly modify $ \sigma_2 $ to get $ s_2 = 0 $ (analogous notation to above).
	Facts \ref{Facts:Id_Exponents} (1) and \eqref{eq:first_step_proof} allow us to keep the new system $ (\sigma ) $ triangular.
	We go on and eventually get
	$$
		\IT(\IE) \, \sim \, (I,b) \cap \bigcap_{i=1}^s ( \sigma_i , q_i) \, \sim \, \bigcap_{i=1}^s ( \sigma_i , q_i) = \Rid ( \IE ) \,,
	$$
	where the second equivalence follows by Facts \ref{Facts:Id_Exponents}(1) and (2) and the assumption that $ ( \sigma )  $ defines the ridge of $ \IE $.
	This shows the first assertion of the proposition and the second is an immediate consequence.
\end{proof}

\begin{Rk}		
	Originally, we used in our proof differential operators of the form $ \frac{\partial}{\partial \tau } $ for some additive polynomial $ \tau $.
	A more general variant of these were introduced in Dietel's thesis \cite{BernhardThesis},
	for the special case mentioned see loc.~cit.~Example (3.3.2).
	Although these are different from Hasse-Schmidt derivations in general, they coincide under favorable conditions (e.g.~if $ (\sigma) $ are in triangular form). 
\end{Rk}

\smallskip 

Let us recall condition $ \boldsymbol{(\cD)} $:
Let $ \IE$ be a pair on $ R $
and let $ ( u, y  ) = (u_1, \ldots, u_e; y_1, \ldots, y_r ) $ be a regular system of parameters for $ R $ such that $ ( y ) $ defines the directrix of $ \IE $.
Denote by $ (U, Y ) $ the corresponding elements in the graded ring $ gr_M(R) \cong k [U,Y] $.
We say condition $\eqref{condD}$
holds for $ ( \IE,  R ) $ 
if the Hasse-Schmidt derivatives $ \dell{Y^M} : gr_M(R) \to gr_M(R) $
lift to Hasse-Schmidt derivatives $ \dell{y^M} : R \to R $ on $ R $ with $ |M| < b $. 

\medskip

\begin{Thm}
\label{Thm:Reduction}
			Let $ R $ be a regular local ring with maximal ideal $ M $ and residue characteristic $ p = \car{R/M} \geq 0 $.
			Let $ \IE_\sim $ be an idealistic exponent on $ R $
			 represented by a pair $ \IE = ( J, b ) $.
			 Let $ ( u, y ) $ be a \RSP for $ R $ such that the system $ ( y) $ determines the directrix of $ \IE $.

			Suppose that $ \Sing( \IE) \neq \varnothing $ and that condition \eqref{condD} holds for $ (\IE, R) $.
			Then we have that
			\begin{equation}
			\label{eq:equiv_IE_IG_Coef}
			\IE \sim \IG \cap \ID_+,
			\ \ \
			\mbox{ 	where }
			\end{equation}  
			\begin{enumerate}
				\item 
				$ \IG = \bigcap\limits_{i=1}^s ( g_i, q_i ) $,
				for some elements $ (g_1, \ldots, g_s ) $ in $  R $ defining the ridge $ \IRid(\IE) $,
				and
				
				\item 
				$ \ID_+ = (I, c) $, for some ideal $ I \subset R $ and $ c \in \IZ_+ $ with $ \ord_{M} (I) > c $.
			\end{enumerate}	
\end{Thm}

\noindent 
Note that (1) implies that $ q_i = p^{d_i} $, for $ d_i \in \IZ_\gqz $ and we may suppose $ d_1 \leq d_2 \leq \ldots \leq d_s $.

\begin{proof}
	In Theorem \ref{Thm:Equiv} we showed $ \IRid( \IE) \sim \ITC(\IE ) $. 
	In the proof we only used Hasse-Schmidt derivatives \wrt $ (Y) $.
	Since \eqref{condD} holds for $ ( \IE, R ) $, we can apply the corresponding Hasse-Schmidt derivatives \wrt $ (y) $ on $ \IE = ( J, b ) $ in $ R $ and obtain using the analogous arguments:
	$$
		\IE \,\sim\, \IG \cap \IE \,\sim\, \IG \cap \ID_+\,,
	$$
	for $ \IG $ and $ \ID_+ $ as in the assertion,
	where we use Facts \ref{Facts:Id_Exponents}(1) and (2) for the second equivalence in order to eliminate the terms corresponding to the initial part in $ \IE $.
\end{proof}

It is worth to mention that the elements $ g_i $ are of the form $ g_i = a_i + h_i $, where $ a_i $ is an additive homogeneous polynomial of degree $ q_i = p^{d_i } $ 
and $ h_i $ are some terms of higher order or zero.

\smallskip

Using the previous result, we can give the following criterion with which we can test if it is possible to reduce the order lowering problem to one in lower dimensions.

\begin{Proc}
\label{CharacterizationProcedure}
	Let the situation and notations be as in Theorem \ref{Thm:Reduction}.
	We have 
	$$ 
		\IE \, \sim \, ( g_1, p^{d_1} ) \cap ( g_2, p^{d_2} ) \cap \ldots ( g_s, p^{d_s} ) \cap \ID_+\,.
	$$
	We apply Facts \ref{Facts:Id_Exponents}(1) for every $ g_i $ which is a $ p $-power and reduce the assigned numbers as much as possible.
	We obtain that
	$$ 
		\IE \, \sim \, ( f_1, p^{c_1} ) \cap ( f_2, p^{c_2} ) \cap \ldots ( f_s, p^{c_s} ) \cap \ID_+ \,,
	$$
	where none of the $ f_i $ is a $ p $-power and after a possible reordering we may assume $ c_1 \leq \ldots \leq c_s $.
	We have the following cases:

	\begin{enumerate}
	\item	
		If $ c_1 > 0 $ then there is no reduction to a (local) resolution problem of lower dimension possible.
		Here we mean a reduction as in the proof for resolution of singularities in characteristic zero.
	
	\medskip

	\item
		If $ c_1 = \ldots = c_t = 0 $, for some $ t \geq 1 $, then without loss of generality $ f_i = y_i $, for $ 1 \leq i \leq t $ and $ V ( y_1, \ldots, y_t ) $ has maximal contact with $ \IE $.
		By Facts \ref{Facts:Id_Exponents}(5), we get that
		$$ 
			\IE \sim ( \langle y_1, \ldots, y_t \rangle, 1 ) \cap \ID (\IE ; u, y_{t+1}, \ldots, y_r;  y_1, \ldots, y_t ) 
		$$
		and the (local) resolution problem reduces to one on $ V (y_1, \ldots, y_t) $. 
		Set $ \ID := \ID (\IE ; u, y_{t+1}, \ldots, y_r;  y_1, \ldots, y_t )  $.
		We have:
		
		\medskip
		\begin{enumerate} 
			\item	If there exists a resolution for $ \ID $ (e.g., if the resolution problem for $ \ID$ corresponds to one in dimension two \cite{CJS}, or if $ \ID $ is defined by binomials \cite{BMbinomial,Blanco1,Blanco2}), then 
			we can achieve an improvement of the singularity.
			More precisely, we can make $ \iota_X $ decrease below $ \iota_X(x_0)_{t+\half } $,
			where $ x_0 $ is the closed point corresponding to $ M $.
			
				\medskip

			\item	Suppose we are not in case (a) and $ t = s $.
					Then we repeat the reduction procedure for the companion pair associated to $ \ID $
					(cf.~Construction \ref{Constr.} and Examples
					\ref{Ex:moresteps1} and \ref{Ex:Not_full_Reso}).

				\medskip
			
			\item 		Suppose we are not in case (a) and $ t < s $.
					Then there is no  further reduction (in the sense of characteristic zero) to a resolution problem in lower dimension possible.
		\end{enumerate}
\end{enumerate}
	
\end{Proc}

\smallskip

Of course, this is not a deep criterion and certainly known to experts.
Nevertheless this is more than just saying ``apply Hironaka's procedure for characteristic zero and see if it works" and hopefully provides an easier access for non-experts who want to determine explicitly a resolution of singularity for special cases in positive or mixed characteristic. 

\smallskip 

As one can see from the above, one of the most interesting singularities to study from the viewpoint of resolution of singularities are hypersurfaces given by some element of the form
$$
	y^{p^e} + h(u_1, \ldots, u_n , y) 
	\ \ 
	\mbox{ with }
	\ h \in \langle u, y \rangle^{p^e + 1}.
$$

For $ e = 1 $, $ n = 3 $ and $ h(u,y) $ either of the form $ h(u,y) = f(u) $ or $ h(u,y) = g^{p-1} y + f(u) $, an explicit resolution of singularities is constructed by Cossart and Piltant \cite{CP2} (over differentially finite fields) and \cite{CPmixed2} (arithmetic case).
Furthermore, they deduce from this the existence of a global birational resolution of singularities in dimension three
(see also \cite{CP1} resp.~\cite{CPmixed} for the reduction to these cases).

For $ n = 3 $, and $ e = 2 $, or $ e = 1 $ and $ h $ not of the above form, the construction of an embedded resolution of singularities for the corresponding hypersurfaces via blowing ups in regular centers remains a difficult open problem. 

%
%
%
%
%
%
%
%		det
%
%
%
%
%
%
%

\medskip

\section{Embedded resolution of singularities for determinantal varieties}

In the final section we discuss how the results of this article provide the ground for the explicit construction of a resolution of singularities for varieties defined by minors of fixed size of a generic matrix.
For an introduction to determinantal varieties, see for example the book by Bruns and Vetter \cite{BrunsVetter} or Harris' book \cite{Harris}, Lecture 9. 
We also refer to \cite{DetermResol} and \cite{DetVarI, DetVarII} for the discussion of a resolution of singularities for these.

%\smallskip

\begin{Def} 
Fix positive integers $ m, n , r \in \IZ_+ $ such that $ r \leq m \leq n $.
Let $ R_0 $ be a regular ring (e.g.~$ R_0 $ a field) and
let $ S = R_0 [x_{i,j} | 1 \leq i \leq m, 1 \leq j \leq n ] $ be the polynomial ring over $ R_0 $ with $ m\cdot n $ independent variables.
Set $ Z := \Spec ( S ) $ and $ M := (x_{i,j})_{i,j} $ be the generic $ m \times n $ matrix defined by the variables. 
The {\em generic determinantal variety} 
$$ X_{m,n,r} = V ( \mathcal{J}_{m,n,r} ) $$ 
is defined as the subvariety of $ Z $ given by the $ r \times r $ minors of $ M $. 

For subsets $ I \subset \{ 1, \ldots, m \} $ and $ J \subset \{ 1, \ldots, n \} $ with $ \#I = \#J = r $, we define
\begin{equation}
\label{eq:f_{I,J}}
	f_{I,J} := \det ( M_{I,J}),
\end{equation}
where $ M_{I,J} $ denotes the $ r \times r $ submatrix of $ M $ determined by $ I $ and $ J $.
By definition,
$$
	\mathcal{J}_{m,n,r} = \langle f_{I,J} \mid 
	I\subset \{ 1, \ldots, m \}, 
	J \subset \{ 1, \ldots, n \} : 
	\#I = \#J = r  \rangle .	
$$
\end{Def} 

\smallskip 

Clearly, each $ f_{I,J} $ is homogeneous of degree $ r $, 
every appearing monomial is a product of $ r $ different variables, and every variable $ x_{i,j} $, with $ i \in J $, $ j \in J $, appears in $ f_{I,J} $.
Further, $ ( f_{I,J})_{I,J} $ is a standard basis for $ \mathcal{J}_{m,n,r} $ at the point corresponding to $ \langle x_{i,j} \mid i, j \rangle $ (Definition \ref{Def:standardbasis(2)}),
for this we may choose any total ordering on the index set
$ \{ (I,J) \mid I\subset \{ 1, \ldots, m \}, 
J \subset \{ 1, \ldots, n \} : 
\#I = \#J = r \} $.
Therefore, we associate the following pair to $ X_{m,n,r} $, 
$$
	\IE_{m,n,r} := \bigcap_{I,J} (f_{I,J} , r ) = (\mathcal{J}_{m,n,r}, r).
$$
We have $ \IRid(\IE_{m,n,r}) =  \bigcap_{i=1}^m  \bigcap_{j=1}^n  ( X_{i,j}, 1 ) = \IDir(\IE_{m,n,r}) $ and hence Theorem \ref{Thm:Reduction} provides.

\begin{Lem}
	\label{Lem:above}
	We have that
	$$ 
		\IE_{m,n,r} \sim  \bigcap_{i=1}^m  \bigcap_{j=1}^n  ( x_{i,j}, 1 ) = \IE_{m,n,1}.
	$$
	In particular, $ \Sing ( \IE_{m,n,r} ) = V ( x_{i,j} \mid 1 \leq i \leq m, 1 \leq j \leq n ) $.
	Hence this is the unique permissible center for $ \IE_{m,n,r} $ of maximal dimension and after blowing up the transform of $ \IE_{m,n,r} $ is resolved..
\end{Lem}

\begin{proof}
	Since $ \IRid(\IE_{m,n,r}) =  \bigcap_{i=1}^m  \bigcap_{j=1}^n  ( X_{i,j}, 1 ) $, we get (using the notation of Theorem \ref{Thm:Reduction})
	$$
		\IG = \bigcap_{i=1}^m  \bigcap_{j=1}^n  ( x_{i,j}, 1 )
		\ \ \mbox{ and } \ \
		\IE_{m,n,r} \sim \IE_{m,n,r} \cap \IG.
	$$
	The desired equivalence follows by Facts \ref{Facts:Id_Exponents}(1) and (2).	
	The remaining parts are immediate consequences of the equivalence.
\end{proof}

\begin{Obs}
	\label{Obs:n,n,n}
	Let us have a look at the case $ m = n = r $.
	Then $ M $ is a square matrix and $ X_{n,n,n} $ is defined by $ f := \det(M) $.
	We blow up with center $ V ( x_{i,j} \mid 1 \leq i, j \leq n ) $.
	In the $ X_{1,1} $-chart the strict transform $ f' $ of $ f $ is given by the following determinant
	$$
	\det \left(
	\begin{matrix}
	1	& x'_{1,2} & \cdots & x'_{1,n} \\
	x'_{2,1} & x'_{2,2} & \cdots & x'_{2,n}\\
	\vdots & \vdots & \ddots & \vdots \\
	x'_{n,1} & x'_{n,2} & \cdots & x'_{n,n}
	\end{matrix}
	\right)
	=
	\det \left(
	\begin{matrix}
	1	& x'_{1,2} & \cdots & x'_{1,n} \\
	0 & x'_{2,2} - x'_{2,1} x'_{1,2}  & \cdots & x'_{2,n} - x'_{2,1}x'_{1,n}  \\
	\vdots & \vdots & \ddots & \vdots \\
	0 & x'_{n,2} -x'_{n,1} x'_{1,2}    & \cdots & x'_{n,n}- x'_{n,1}x'_{1,n}
	\end{matrix}
	\right)
	=
	$$ 
	$$
	=
	\det \left(
	\begin{matrix}
	y_{2,2}   & \cdots & y_{2,n}   \\
	\vdots & \ddots & \vdots \\
	y_{n,2}    & \cdots & y_{n,n}
	\end{matrix}
	\right) \,,
	\mbox{ where } y_{i,j} := x'_{i,j} - x'_{i,1} x'_{1,j} \,, \mbox{ for } 2 \leq i, j \leq n \,.
	$$ 
	(Note that the situation in the other charts is analogous).
	As we see, $ f' $ is now defined by a generic $ ( n - 1) \times ( n - 1 ) $ matrix.
	In other words, the strict transform $ X'_{n,n,n} $ of $ X_{n,n,n} $ is isomorphic to $ X_{n-1,n-1,n-1} $.
	The previous lemma provides that the next center in this chart is $ V ( y_{i,j} \mid 2 \leq i, j \leq n ) $.

	We obtain that after $ n - 1 $ blowing ups following this procedure the strict transform in each chart will be resolved since it is isomorphic to $ X_{1,1,1} $, i.e., it is given by a generic $ 1 \times 1 $ matrix.	
But we have to convince us that this yields a global resolution. 
Namely, we have to show that the centers of the blowing ups in each chart glue together to a global center.

By definition, $  y_{i,j} = x'_{i,j} - x'_{i,1} x'_{1,j} $ and hence it is the strict transform of $ x_{1,1} x_{i,j} - x_{i,1} x_{1,j} $.
More generally, $ \langle y_{i,j} \mid 2 \leq i, j \leq n  \rangle $ is the strict transform of the ideal defined by the $ 2 \times 2 $ minors of $ M $,
$ 
	\langle x_{a,b} x_{i,j} - x_{i,b} x_{a,j} \mid 1 \leq a,b,i, j \leq n  \rangle = \mathcal{J}_{n,n,2} .
$
Note that, for $ i \neq 1, j \neq 1, b \neq 1 $ we have in the $ X_{1,1} $-chart
$$
	x'_{1,b} x'_{i,j} - x'_{i,b} x'_{1,j} 
	=  x'_{1,b} (y_{i,j} + x'_{i,1} x'_{1,j}) -  (y_{i,b} + x'_{i,1} x'_{1,b}) x'_{1,j} 
	= x'_{1,b} y_{i,j}  -  y_{i,b} x'_{1,j}
$$
and hence this vanishes in $ V ( y_{i,j} \mid 2 \leq i, j \leq n ) $.

This implies that the upcoming centers after the first blowing up globally glue together and hence define a global center. 
By iterating this argument we obtain that at every step of the described procedure the centers glue together and hence define a global center.
In fact, the center of the $ i $-th blowing up is the strict transform of $ X_{n,n,i} $ which is defined by the $ i \times i $ minors of $ M $. 
This shows that we get a global resolution of the singularities of $ X_{n,n,n} $.
\end{Obs}

It turns out that these methods apply in the case for arbitrary $ r \leq m \leq n $, see below.
Therefore, using the techniques developed in this article, we obtain a slight generalization of Vainsencher's result \cite{DetermResol}, who obtains a resolution via blowing up determinantal ideals if $ R $ contains an algebraically closed field $ k $.
(Indeed, his centers of the blowing ups are defined by the same ideals as our's).

\begin{Thm}
	\label{Thm:ResDetText}
	Let $ m , n, r \in \IZ_+ $ with $ r \leq m \leq n $, let $ R_0 $ be a regular ring, and let $ M = (x_{i,j})_{i,j} $ be a generic $ m \times n $ matrix.
	The following sequence of blowing ups provides an embedded resolution of singularities for the generic determinantal variety $ X_{m,n,r} \subset Z = \Spec(R_0[x_{i,j}]) $,
	$$ 
	Z =: Z_0 \stackrel{\pi_1}{\longleftarrow} Z_1
	\stackrel{\pi_2}{\longleftarrow} 
	\ldots 
	\stackrel{\pi_{r-1}}{\longleftarrow} Z_{r-1}, 
	$$  
	where $ \pi_\ell $ is the blowing up with center the strict transform of $ X_{m,n,\ell} $ in $ Z_{\ell-1} $, $ 1 \leq \ell \leq r - 1 $.
\end{Thm}

\begin{proof}
	We prove the result by induction on $ r $.
	For $ r = 1 $, there is nothing to show since $ X_{m,n,1} $ is regular.
	Suppose that $ r > 1 $ and that the statement holds true for $ X_{m,n,i} $ if $ i < r $.
	In particular, we obtain that each center of the above sequence of blowing ups is regular.
	
	Let $ i_1 \in \{ 1, \ldots, m \} $ and $ j_1 \in \{ 1, \ldots, n \} $.
	Let $ N $ be the $ (m-1)\times (n-1) $ matrix that we obtain after eliminating the $ i_1 $-row and the $ j_1 $-th column in $ M $, and then replacing 
	$ x_{i,j} $ by
	$$ 
		y_{i,j} :=  x'_{i,j} - x'_{i,j_1} x'_{i_1,j} ,
		\ \ \ 
	\mbox{ for } i \neq i_1 \mbox{ and } j \neq j_1 .
	$$
	
	We claim that the strict transform of $ X_{m,n,r} $ in the chart $ X_{i_1,j_1} \neq 0 $ coincides with the generic determinantal variety $ Y_{m-1,n-1, r-1} $ that is defined using the matrix $ N $.
	By induction, we obtain then that the above sequence resolves $ X_{m,n,r} $ and since the blowing ups are defined globally (cf.~Observation \ref{Obs:n,n,n}), the result follows.
	
	We prove the claim:
	Let $ ( f_{I,J} )_{(I,J)} $ be the standard basis for $ \mathcal{J}_{m,n,r} $ defined in \eqref{eq:f_{I,J}}.
	(Recall that the strict transform of a standard basis generates the strict transform of the corresponding ideal under a permissible blowing up).
	If $ i_1 \in I $ and $ j_1 \in J $, then the same computation as at the beginning of Observation \ref{Obs:n,n,n} applied for $ f_{I,J} $ provides that 
	$$
		\mathcal{I} :=  \langle f_{I,J}' \mid i_1 \in I, j_1 \in J \rangle 
	$$ 
	coincides with the ideal defining $ Y_{m-1, n-1, r- 1} $.
	By Lemma \ref{Lem:above}, $ \Sing(Y_{m-1,n-1,r-1}) = Y_{m-1,n-1,1} $ and the same arguments as in Observation \ref{Obs:n,n,n} provide that this is the transform $ X_{m,n,2}' $ of the $ 2 \times 2 $ minors of $ M $ in the present chart. 
	
	Hence it remains to show that $ f_{I,J}' \in \mathcal{I} $ if $ i_1 \notin I $ or $ j_1 \notin J $.
	Suppose $ i_1 \notin I $ and $ j_1 \in J $.
	Without loss of generality, $ j_1 $ is the first column of $ M_{I,J} $ (interchanging columns only modifies the sign of the determinant).
	Let $ I = \{ k_1, \ldots, k_r \} $ and $ J = \{ j_1, \ell_2, \ldots, \ell_r \} $.
	Using the definition of $ y_{i,j} $, we get that
		$$
		f_{I,J}' =
		\det \left(
		\begin{matrix}
		x'_{k_1,j_1} & x'_{k_1,\ell_2} & \cdots & x'_{k_1,\ell_r} \\
		x'_{k_2,j_1} & x'_{k_2,\ell_2} & \cdots & x'_{k_2,\ell_r}\\
		\vdots & \vdots & \ddots & \vdots \\
		x'_{k_r,j_1} & x'_{k_r,\ell_2} & \cdots & x'_{k_r,\ell_r}
		\end{matrix}
		\right)
		=
		$$
		
		\medskip
		
		$$
		=
		\det \left(
		\begin{matrix}
		x'_{k_1,j_1} & y_{k_1,\ell_2} + x'_{k_1,j_1} x'_{i_1,\ell_2} & \cdots &  y_{k_1,\ell_r} + x'_{k_1,j_1} x'_{i_1,\ell_r} \\
		x'_{k_2,j_1} & y_{k_2,\ell_2} + x'_{k_2,j_1} x'_{i_1,\ell_2} & \cdots &  y_{k_2,\ell_r} + x'_{k_2,j_1} x'_{i_1,\ell_r} \\
		\vdots & \vdots & \ddots & \vdots \\
		x'_{k_r,j_1} & y_{k_r,\ell_2} + x'_{k_r,j_1} x'_{i_1,\ell_2} & \cdots &  y_{k_r,\ell_r} + x'_{k_r,j_1} x'_{i_1,\ell_r} 
		\end{matrix}
		\right)
		=
		$$
		
		\medskip
	
		$$
		=
		\det \left(
		\begin{matrix}
		x'_{k_1,j_1} & y_{k_1,\ell_2} & \cdots &  y_{k_1,\ell_r} \\
		x'_{k_2,j_1} & y_{k_2,\ell_2}  & \cdots &  y_{k_2,\ell_r}  \\
		\vdots & \vdots & \ddots & \vdots \\
		x'_{k_r,j_1} & y_{k_r,\ell_2} & \cdots &  y_{k_r,\ell_r}  
		\end{matrix}
		\right)
		\in \mathcal{I}. 
		$$ 
		In the last equality, we apply elementary column operations and in order to see that the last element is contained in $ \mathcal{I} $, we expand the determinant with respect to the first column.
	The arguments in the other cases are similar and we leave the details to the reader.	 
\end{proof}

\smallskip

In general, a determinantal variety is not necessarily defined via a generic matrix.
One may allow arbitrary polynomial entries for $ M = (f_{i,j})_{i,j} $ with $ f_{i,j} \in R_0 [x_1, \ldots, x_N ] $
($ R_0 $ as before).
For recent results studying these singularities from various viewpoints (over a field of characteristic zero),
we refer to \cite{BruceTari,Anne16,AnneNeumer,AnneMatthias,GaffneyToni,GoryMond,PereiraRuas}.
An interesting problem would be to investigate which assumptions need to be imposed on $ R_0 $ and the entries $ f_{i,j} $ in order to construct a characteristic-free embedded resolution for this type of singularities.
One new technical detail that one has to take care of is then the choice of a standard basis for the defining ideal which has been immediate in the generic case.

\medskip

%
%
%
%
%
%
%
%		Literature
%
%
%
%
%
%
%

\label{Literature}

\newcounter{zaehler}
\setcounter{zaehler}{1}
\newcommand{\art}{\thezaehler}
 
%\bibliography{References}

\end{document}